\documentclass[10pt]{amsart}
\usepackage{amsmath}
\usepackage{paralist}
\usepackage{graphicx} 
\usepackage{epstopdf}
\usepackage[colorlinks=true]{hyperref}
\usepackage{float}
\usepackage{mathtools}
\usepackage{color,comment}
\usepackage{caption}
\usepackage{wrapfig}
\usepackage{multirow}
\usepackage{tabularx}
\usepackage{booktabs}
\usepackage{subfigure}
\usepackage{chngcntr}
\usepackage{enumitem}
\newtheorem{theorem}{Theorem}[section]

\newtheorem{lemma}[theorem]{Lemma}

\newtheorem{conjecture}{Conjecture}

\theoremstyle{definition}

\newtheorem{example}{Example}
\newtheorem{remark}{Remark}


\begin{document}
\title[Non-smooth Model with Fear and Mutual Interference]{Fear-Driven extinction and (de)stabilization in a predator-prey model incorporating prey herd behavior and mutual interference}

\author[Antwi-Fordjour, Parshad, Thompson, Westaway]{}
\maketitle

\centerline{\scshape Kwadwo Antwi-Fordjour$^a$, Rana D. Parshad$^b$, Hannah E. Thompson$^c$}
\centerline{\scshape and Stephanie B.  Westaway$^a$}

\medskip
{\footnotesize
  \centerline{ a) Department of Mathematics and Computer Science,}
 \centerline{Samford University,}
 \centerline{Birmingham, AL 35229, USA}
\medskip
{\footnotesize
 \centerline{ b) Department of Mathematics,}
 \centerline{Iowa State University,}
\centerline{Ames, IA 50011, USA}
}
 \medskip
 {\footnotesize
  \centerline{ c) Department of Biological and Environmental Sciences,}
 \centerline{Samford University,}
 \centerline{Birmingham, AL 35229, USA}
}

\begin{abstract}
  A deterministic two-species predator-prey model with prey herd behavior is considered incorporating mutual interference and the effect of fear. 
We provide guidelines to the dynamical analysis of biologically feasible equilibrium points. We give conditions for the existence of some local and global bifurcations at the coexistence equilibrium. We also show that fear can induce extinction of the prey population from a coexistence zone in finite time. Our numerical simulations reveal that varying the strength of fear of predators with suitable choice of parameters can stabilize and destabilize the coexistence equilibrium solutions of the model. Additionally, we discuss the outcome of introducing a constant harvesting effort to the predator population in terms of changing the dynamics of the system, in particular, from finite time extinction to stable coexistence.  Furthermore,  we perform extensive numerical experiments to visualize the dynamical behavior of the model and substantiate the results we obtained. 
\end{abstract}

\section{Introduction}
\noindent 
Interactions among predator and prey population species are modeled by systems of differential equations, and the functional response (number of prey consumed per predator per unit of time) of predators toward the prey is one of the important ecological components, which provides a bedrock for predator-prey dynamics. Mathematical models incorporating a functional response originated from the investigation of chemical reactions and biological interactions  \cite{L25,V26}.  A large body of scholarly literature has shown that the functional response of predator can have profound impacts on the dynamics in natural predator-prey communities \cite{G34,B75,D75,H59,P15,H65,A02,UR97,KRM20}. Thus, to make mathematical models more realistic, an appropriate choice of functional response is needed. 

\emph{Herd behavior} refers to the phenomenon in which individuals in a group act collectively for a given period without coordination by a central authority \cite{B92,L06,RCF09}. This behavioral phenomenon has been widely researched across several disciplines. For example, in early economics, Veblen studied herd behavior in sudden shifts in consumer behavior, such as fads and fashions  \cite{V99}. Again, such phenomenon is seen in the Cobb-Douglas production function in the econometrics literature \cite{C28}. Its effect on population dynamics can be modeled using a functional response.   Prey herd behavior is a form of anti-predator behavior and provides protection for the prey species against predators \cite{APV11,VGR18}. One way of modeling prey heard behavior is by using the functional response, $\varphi(u)=cu^p$, with $0<p<1$ introduced by Rosenzweig \cite{RZ71}, where $u=u(t)$ is the population density of the herd and $c$ is the predation rate. Symbiotic, competition and predator-prey models in which the interaction terms use square root of the density of one population was considered by Ajraldi et al. \cite{APV11}. Furthermore,  Braza \cite{B12} studied a predator-prey model with the square root functional response $\varphi(u)=cu^{1/2},$ proposed by Gauss \cite{G34}, implying a strong herd behavior where the predator interacts with the prey along the outskirts of the herd. An ecoepidemic predator-prey model with feeding satiation showing prey herd behavior and abandoned infected prey was investigated by Kooi and Venturino \cite{K16}.  For other ways of modeling prey herd behavior, see \cite{H59,KRM20} and references therein. 
 
In addition, finite time extinction (FTE) of species exists in ecosystems and it is a significant issue for the management of natural resources \cite{RZ71,RKM20}. The functional response $\varphi(u)=cu^p$, for $p<1$ is non-smooth for $u=0$ and therefore possess interesting complex dynamics. This response function allows for the extinction of the prey  species in finite time, after which the predator population species exponentially decay to zero in infinite time \cite{AR95,SS07,B12,VGR18,BS19}. Non-smooth functional responses (or power incidence functions) have been analyzed in susceptible-infective models, where the host species can potentially go extinct in finite time \cite{FCGT18}.


The need to consider intra-specific behavioral interactions among predators when searching for prey is a vital question for ecologists and conservationist trying to ascertain the dynamics that inform ecosystem balance. These behavioral effects, also known as mutual/predator interference impede the predators' searching efficiency as the density of the predators increases \cite{H69,H71,H75,F79}. Several studies have concluded that mutual interference has a stabilizing effect on population dynamics, see \cite{A04} and references therein.    Freedman and Wolkowicz \cite{F86} investigated the survival or extinction of predators in a deterministic predator-prey system exhibiting prey group defence. In this study, they determined that extinction due to group defence combined with enrichment can be prevented by introducing mutual interference of predators.
 For further discussions of mutual interference with other types of functional response, see \cite{E85,U19,L09,K09,M13,W13}.  \par

 Recently, the non-consumptive effects of predation due to fear of predators has become the subject of interest for ecologists and mathematical biologists. Experiments on terrestrial vertebrates showed that the presence of a predator may play an important role by changing the behavior of the prey demography \cite{Zanette11}. Zanette et al. \cite{Zanette11} manipulated predation risk of song sparrows for the duration of an entire breeding season. This experiment was conducted to ascertain whether perceived predation risk alone could have an impact on the reproduction of the song sparrows. Suraci et al. \cite{S16} observed from experimentation that the effect of the manipulation of fear of large carnivores causes a tropic cascade. Hua et al. \cite{Hua14} studied how increased perception of predation risk to adults and offspring alters reproductive strategy and performance. In Wirsing and Ripple \cite{Wirsing11}, a comparison of shark and wolf research revealed similar behavioral responses by prey. Bauman et al. \cite{B19} investigated how the effects of fear associated with predator presence and habitat structure interact to change the removal of macroalgal biomass (i.e \emph{herbivory}) on coral reefs. They observed that the effects of fear due to the presence of predators were highest at low macroalgal density, but lost at higher densities due to increased background risk. ~\par
 
 Motivated by these ecological and biological findings, Wang et al. \cite{Wang16} introduced a mathematical model incorporating fear. The authors demonstrated that strong fear responses can have a stabilizing effect on a predator-prey model with Holling type II response by excluding periodic solutions to the system, resulting in a locally stable point of coexistence between the predator and prey populations \cite{Wang16}. Subsequent studies investigated the dynamics of fear in models incorporating hunting cooperation, prey refuge, Leslie-Gower type and a variety of functional responses, such as Beddington-DeAngelis, Holling type I,II, III and IV \cite{Kumar19, Zhang19, Pal19, Pal&Pal19, ST20}. A recent work considering the effect of mutual interference and fear on a predator-prey model with a Holling type I functional response established that the inclusion of mutual interference promotes system stability \cite{Xiao19}. ~\par

The qualitative effect of predator harvesting on the stability of the ecosystem has been investigated extensively \cite{B79,K06,X06,D98,L16,F19}. Chakraborty et al. \cite{C11} explored a mathematical study with biological ramifications of a predator-prey model with predator effort harvesting. Their result suggests that harvesting of predator may be one of several ways to observe coexistence of prey and predator population species in the laboratory study and possibly nature. The ratio-dependent predator-prey model where the predator population is harvested at catch-per-unit-effort hypothesis is investigated by Gao et al. \cite{G20}. Therein, they studied the temporal, spatial and spatiotemporal  rich dynamics due to the non-smoothness of the origin.

Our primary contributions in the present manuscript are:
\begin{enumerate}
\item We formulate a mathematical model (i.e. model \eqref{EquationMain}) incorporating the combined effects of fear of predator, prey herd behavior and mutual interference.
\item We study the effect of fear of predators on the dynamics of the model \eqref{EquationMain}.  We note that when there is no fear (i.e. $k=0$), there is some initial data that converges uniformly to a stable coexistence equilibrium point. With the introduction of fear of predator (i.e. $k>0$), that same initial data will converge to the predator axis in finite time. This phenomenon is shown analytically via Theorem \ref{thm:FDFTE} and presented numerically in Fig. \ref{fig:FDFTE}.
\item We analyze the impact of  predator harvesting on the dynamics of the modified Lotka-Volterra  model with fear effect. Our mathematical conjecture (see Conjecture \ref{thm:harvesting induced recovery}) and numerical simulation (see Fig. \ref{fig:Harvest_Recovery_Homoclinic}) reveal that, harvesting of predators can prevent finite time extinction of the prey species.
\end{enumerate}

This paper is arranged as follows: In Section \ref{Sect:model_formulation}, we propose a mathematical model of systems of differential equations to incorporate the combined effects of fear of predators, prey herd behavior and mutual interference. Guidelines to dynamical analysis are presented in Section \ref{Sec:dynamical_analysis}, where we investigated the possible existence of biologically feasible equilibrium points and the stability of the coexistence equilibrium. In Section \ref{Sect:bifurcation_analysis}, we derive conditions for the existence of  local and global bifurcations including saddle-node, Hopf, Bautin, and homoclinic bifurcations. Finite time extinction of the prey species driven by fear of predators were analyzed in Section \ref{Sec:FDFTE}. In Section \ref{Sect:Harvesting}, we investigate the effect of effort harvesting of predators.  To illustrate the feasibility of our mathematical analysis and conjectures, extensive numerical solutions are presented. The paper ends with concluding remarks in Section \ref{Sec:Conclusion}.

\section{The Mathematical Model}\label{Sect:model_formulation}
\noindent
Consider a modified Lotka-Volterra predator-prey model with predation intensity  and mutual interference of predators. Let $u(t)$ and $v(t)$, respectively, denote the prey and predator population densities at any time $t$. The model is given by the following systems of equations

\begin{equation}\label{eqn:Lotkta-Volterra}
\begin{cases}
\dfrac{du }{dt} &=a u  - b u^2 -  c u^p v^m,\qquad u(0)\geq 0, \\
\dfrac{dv }{dt} &=-d  v + e u^p v^m, \qquad \qquad v(0)\geq 0,
\end{cases}
\end{equation}
where $0<m,p\leq 1$. Let $m$ denote mutual interference parameter introduced by Hassell \cite{H71}, $1/p$ is the intensity of predation and $p$ determines the slope of the functional response at the origin,  $a$ denotes the birth rate of prey, $d$ denotes the death rate of predator, $b$ reflects the intraspecific competition of the prey, $c$ denotes the rate of predation, and $e$ measures efficiency of biomass conversion from prey to predator. When $p=m=1$, the model \eqref{eqn:Lotkta-Volterra} degenerates to the classical classic Lotka-Volterra model \cite{L25,V26}.  All parameters are assumed to be positive.
The underlying assumptions of the model \eqref{eqn:Lotkta-Volterra} are as follows:
\begin{enumerate}[label=(\roman*)]
 \item The first equation in model \eqref{eqn:Lotkta-Volterra} describes the change in prey population with respect to time, and it is separated into three parts, namely  birth rate, effect of the density of one species on the rate of growth of the other and functional response of the predator towards the prey.
 \item The second equation in model \eqref{eqn:Lotkta-Volterra} describes the change in predator population with respect to time and it is separated into two parts, namely death rate, $d$, the predators die out in the absence of its only food source, prey and biomass conversion from prey to predator with rate $e$. 
 \item The term $v^m$ models the intra-specific behavioral interactions among predators when searching for prey. For $m<1$, there is predator interference, where larger predator densities leads to less consumption per capita. Furthermore, this leads to a nonvertical predator nullcline. 
 \item The predator is consuming the prey with the functional response $\varphi(u)=u^p$,  for $0<p<1$ (see \cite{KRM20} for assumptions of $\varphi$). 
\end{enumerate}
 Model \eqref{eqn:Lotkta-Volterra} has been well investigated in infectious disease modeling, where $u$ represents the density of susceptible populations, $v$ represents the density of infective populations and the term $u^pv^m$ (i.e. modified Lotka-Volterra interaction term) represents power incidence function \cite{FCGT18}. As we have seen, the  model considered in \cite{FCGT18} reveals some significant and interesting results due to the power incidence function. A natural question that arises is: how does the effect of fear of predators affect the dynamical behaviors of the model \eqref{eqn:Lotkta-Volterra}? Does it stabilize, destabilize or have no influence?\\
Now, based on experimental evidence \cite{Zanette11}, we assume fear of predators decreases the birth rate of the prey species. To account for the decrease in the prey population due to fear of predators, the birth rate of the prey is multiplied by the term $\phi(k,v)=\frac{1}{1+kv}$, introduced by Wang et al. \cite{Wang16}, which is monotonically decreasing in both $k$ and $v$. Here $k$ denotes the strength of fear of predator. Biologically, it is appropriate to assume the following:
\begin{align*}
&\phi(k,0)=1,\quad \phi(0,v)=1,\quad \lim_{v\rightarrow\infty}\phi(k,v)=0,\\ 
&\quad \lim_{k\rightarrow\infty}\phi(k,v)=0, \quad
\dfrac{\partial\phi(k,v)}{\partial v}<0 ,\quad \dfrac{\partial\phi(k,v)}{\partial k}<0. 
\end{align*}
To the best of our knowledge, there does not exist any scholarly literature that investigates the combined influence of fear of predators on predator-prey interactions with prey herd behavior and mutual interference. This motivates us to formulate the following model
\begin{equation}\label{EquationMain}
\begin{cases}
\dfrac{du }{dt} &=\dfrac{a u}{1+k v}  - b u^2 -  c u^p v^m \coloneqq F(u,v), \qquad u(0)\geq 0,\\
\dfrac{dv }{dt} &=-d  v + e u^p v^m \coloneqq G(u,v), \qquad\qquad\qquad v(0)\geq 0.
\end{cases}
\end{equation}
 When $p=m=1$, we recover the results from Wang et al.  \cite{Wang16}. Moreover, for $p=1$ and $0<m\leq 1$, we recover results from Xiao and Li \cite{Xiao19}. Recently, Fakhry and Naji \cite{FN20} investigated the model \eqref{EquationMain}, where the fear function was multiplied to the logistic growth term i.e. $(au-bu^2)\phi(k,v)$ with square root functional response (i.e. $~p=0.5$) and no predator interference (i.e. $m=1$). Huang and Li \cite{HL20} disproved and also provided an alternative proof for some of the results obtained by Fakhry and Naji.  Our model provides a generalization of the models mentioned above, and we will focus on the case where $0<m,p<1$.

\section{Dynamical Analysis}\label{Sec:dynamical_analysis}
\noindent
The dynamical analysis of the model \eqref{EquationMain} is investigated in this section.
\begin{lemma}\label{lemma:Nonnegativity}
Consider the first quadrant $\mathbb{R}_+^2=\{(u,v):u\geq 0, v\geq 0\}$, then the solutions $(u(t),v(t))$ of model the \eqref{EquationMain} which initiate in $\mathbb{R}_{++}^2$ are nonnegative for all $t\geq 0$. Here, $\mathbb{R}_{++}^2=\{(u,v):u> 0, v> 0\}$. 
\end{lemma}
\begin{proof}
The right hand side of model \eqref{EquationMain} is continuous and locally non-smooth in $\mathbb{R}_{+}^2$. Also, the solution $(u(t),v(t))$ which initiate in $\mathbb{R}_{++}^2$ of model \eqref{EquationMain} exists and is non-unique.  From model \eqref{EquationMain}, we obtain
\begin{align*}
u(t)=&u(0)\exp\left[\displaystyle\int_0^t \left( \dfrac{a}{1+kv}  -b u - c u^{p-1}v^{m}  \right)ds \right]\geq 0, \\
v(t)=&v(0)\exp\left[\displaystyle\int_0^t \left( -d + e u^{p}v^{m-1}  \right)ds \right]\geq 0
\end{align*}
However,  $v$ stays positive for all $t>0$.
\end{proof}
\noindent
In theoretical biology and ecology, nonnegativity of the model \eqref{EquationMain} implies survival of the populations over some temporal domain.
\begin{lemma}\label{lemma:Boundedness_Dissipative}
The solutions $(u(t),v(t))$ of model the \eqref{EquationMain} which initiate in $\mathbb{R}_{++}^2$ are  uniformly bounded and dissipative. 
\end{lemma}
The proof of Lemma \ref{lemma:Boundedness_Dissipative} is standard and therefore omitted in this work.\\

\noindent
The boundedness of a system limits total population growth of the interacting species, ensuring that neither population experiences exponential growth over a long time interval. As a condition of this property, total population values will not reach impracticable quantities in a period of time. Also,  in a dissipative model, the population of each species is bounded from above for all time. This guarantees that the individual populations of the predator or the prey do not exceed a finite upper limit.

\subsection{Existence of Equilibria}
The  model \eqref{EquationMain} contains one trivial equilibrium point $E_0(0,0)$, and an axial equilibrium point, $E_1(a/b,0)$.
 The existence of a coexistence equilibrium point is ascertained by finding the intersection(s) of the prey and predator nullclines.\\
  First, the prey nullcline is determined by the equation

\begin{align}
g_1(u,v)=\dfrac{a}{1+kv}  -b u - c u^{p-1}v^{m}  &= 0.  \label{int1} 
\end{align}
 If $u=0$, we obtain
$$0=\dfrac{a}{1+kv}\coloneqq \sigma(v). $$ 
But $v=0$ implies $\sigma(v)=a\neq 0$, since $a$ is a positive constant. Furthermore,
$$\sigma^{\prime}(v)=-\dfrac{ak}{(1+kv)^2}<0,$$ 
and thus, there exist some $v>0$ such that $\sigma(v)=0$.
 Also, if $v=0$ in equation \eqref{int1}, then $u=\frac{a}{b}>0$.
Moreover, we assume that 
\begin{align}\label{cond_on_trace}
b>c(1-p)u^{p-2}v^m
\end{align}
  and observe that $u>0,v>0$ and 

\begin{align*}
\dfrac{dv}{du}=-\left[\dfrac{b-c(1-p)u^{p-2}v^m}{cmu^{p-1}v^{m-1}-\sigma^{\prime}(v)}\right]<0.
\end{align*}
%
Now the graph of the prey nullcline is concave and intersects the prey axis at $u=0$ and $u=\frac{a}{b}$. \\
Additionally, the predator nullcline is determined by the equation 
\begin{align} \label{int2}
g_2(u,v)=-d + e u^{p}v^{m-1}  &= 0.  
\end{align}
Solving for $v$ in equation \eqref{int2} yields
\begin{align}\label{int3}
v=\left[\dfrac{e}{d}u^{p}\right]^{1/(1-m)}
\end{align}
Clearly, the point $(0,0)$ lies on the predator nullcline. By computing the first and second derivatives with respect to $u$, we obtain
$$\dfrac{dv}{du}=\dfrac{pv}{(1-m)u}>0,$$
\begin{align}\label{concavity}
\dfrac{d^2v}{du^2}=\dfrac{p(p+m-1)v}{[(1-m)u]^2}\;(> 0\;\text{or}=0\; \text{or}\; < 0).
\end{align}
From equation \eqref{concavity}, the sign of the second derivative depends on $p+m-1$.

Hence, the predator nullcline goes through $(0,0)$, and as $u$ increases, the predator nullcline increases monotonically. Now by the intermediate value theorem, the prey and predator nullclines will intersect in $\mathbb{R}_{++}^2$ to produce a unique (i.e. $E_2(u^*,v^*)$) or two (i.e. $E_2^i(u_i^*,v_i^*)$, for $i=1,2$ and $0<u_1^*<u_2^*<\frac{a}{b}$) coexistence equilibrium points.

\begin{remark}\label{remark: equilibrium points}
Indeed in  model \eqref{EquationMain}, the coexistence equilibria  are not analytically accessible.   Numerical simulations are provided as guidelines in Fig. \ref{fig:1and2_intEquilibrium}, to show the existence of a unique coexistence  equilibrium point for $p+m\geq1$ and two coexistence  equilibria  for $p+m<1$.
\end{remark}
\begin{figure}[!htb]
\begin{center}
\subfigure[]{
    \includegraphics[scale=.47]{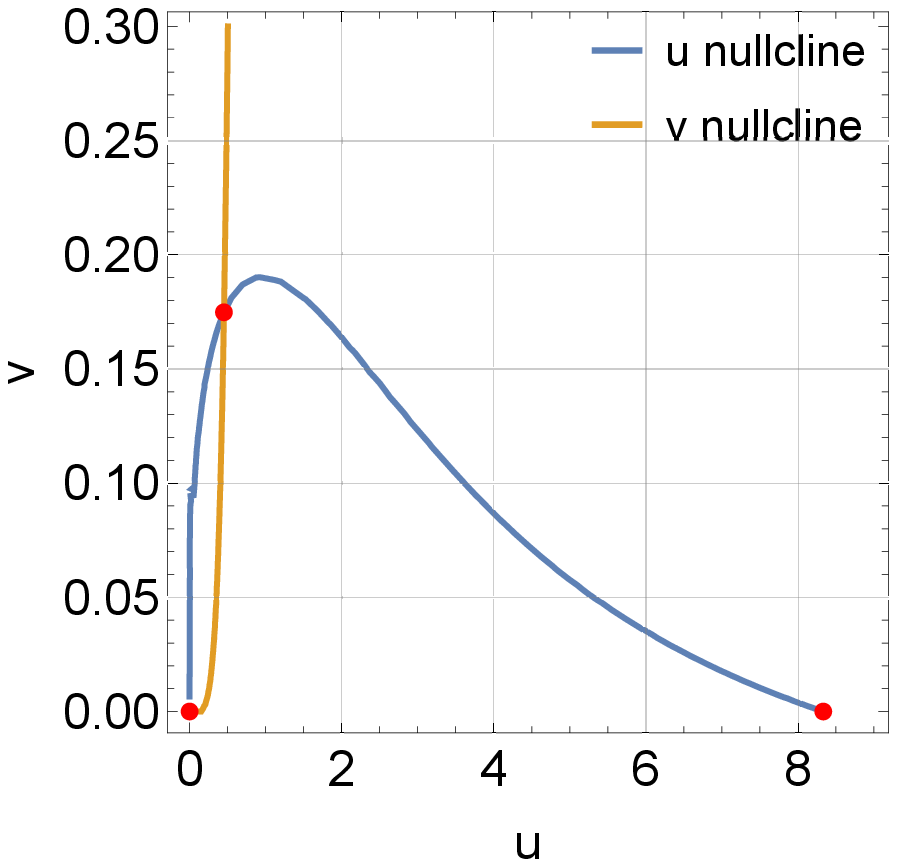}}
\subfigure[]{
    \includegraphics[scale=.43]{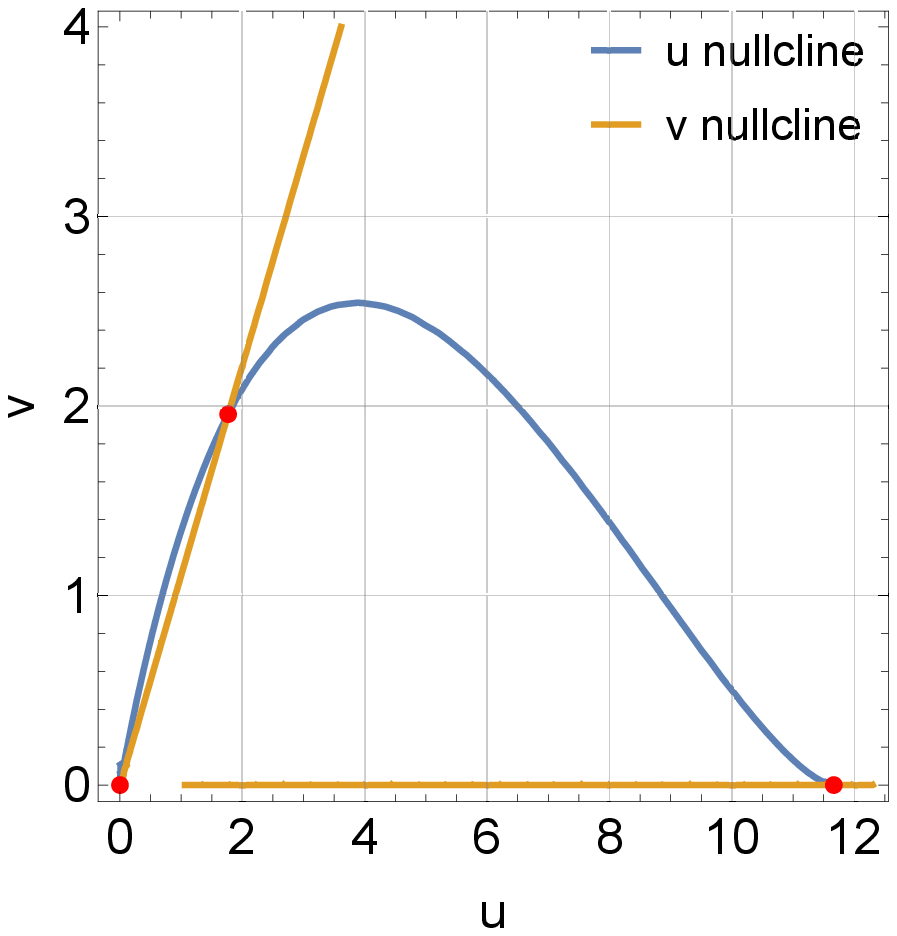}}    
\subfigure[]{    
    \includegraphics[scale=.45]{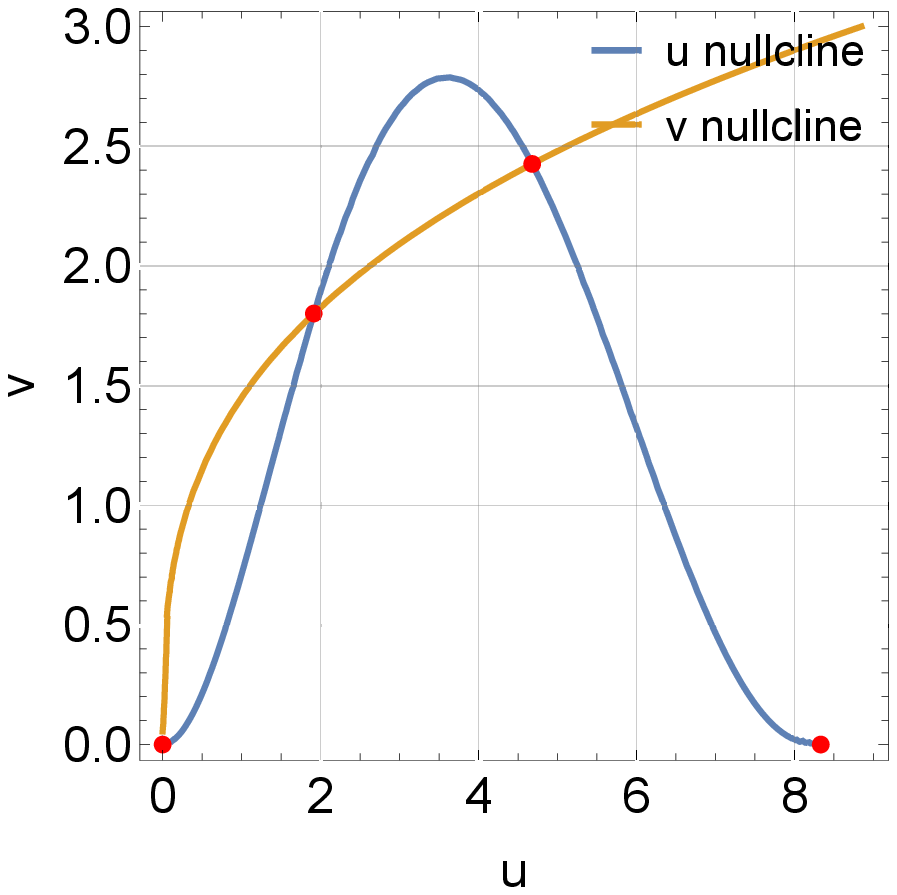}}
\end{center}
\caption{Phase plane portraits depicting the predator and prey nullclines in the model \eqref{EquationMain} (a) unique coexistence equilibrium point for $p=0.5$ and $m=0.9$, thus $p+m>1$ (b) unique coexistence equilibrium point for $p=0.4$ and $m=0.6$, thus $p+m=1$   (c) two coexistence equilibria for  $p=0.2$ and $m=0.4$, thus $p+m<1$. Solid red circles represent equilibrium points.}
\label{fig:1and2_intEquilibrium}
\end{figure}
\subsection{Stability Analysis at a Coexistence Equilibrium Point}
We discuss in this subsection the local stability at any  coexistence equilibrium point. 
\begin{theorem}\label{localstability}
Consider the model given by \eqref{EquationMain}.
\begin{itemize}
\item[(a)] For  $p+m\geq 1$, there exists a unique  coexistence equilibrium point $E_2(u^*, v^*)$ which is locally asymptotically stable (LAS) by the Routh-Hurwitz criterion.
\item[(b)] For  $p+m<1$, either there exist two coexistence equilibrium points or none. However, if there exist two coexistence equilibrium points, then $E^1_2(u_1^*, v_1^*)$ is a saddle and  $E^2_2(u_2^*, v_2^*)$ is LAS. 
\end{itemize}
\end{theorem} 

\begin{proof}
The linearized model \eqref{EquationMain} at any coexistence equilibrium point $(u^*,v^*)$ is given by the Jacobian Matrix $\bf{J}$
\begin{align}\label{Jacobian}
 \bf{J} &= 
     \begin{bmatrix}
     J_{11} &  J_{12}\\
      J_{21} &  J_{22}
     \end{bmatrix},
\end{align}
where\\
\begin{align*}
J_{11}&= \dfrac{a}{1+kv^*}-2bu^*-cpu^{*p-1}v^{*m}\\
&=-bu^*+c(1-p)u^{*p-1}v^{*m} \\
&=-u^*\left[b-c(1-p)u^{*p-2}v^{*m}\right]= u^*\frac{\partial g_1}{\partial u^*}.
\end{align*}
\begin{align*}
J_{12}&=-\dfrac{kau^*}{(1+kv^*)^2}-cmu^{*p}v^{*m-1}=u^*\frac{\partial g_1}{\partial v^*} < 0,\\
J_{21}&= epu^{*p-1}v^{*m}= v^*\frac{\partial g_2}{\partial u^*}  > 0, \\
J_{22}&=  -d+meu^{*p}v^{*m-1}=-d(1-m)= v^*\frac{\partial g_2}{\partial v^*} <0. 
\end{align*}
The characteristic equation at the coexistence equilibrium is 
 $$\eta^2 - \operatorname{tr} \,({\bf{J}}) \eta + \det \,({\bf{J}}) =0,$$
where 
\begin{align*}
 \operatorname{tr} \,({\bf{J}})=J_{11}+J_{22}=u^*\frac{\partial g_1}{\partial u^*}+v^*\frac{\partial g_2}{\partial v^*} ,   
\end{align*}

and
\begin{align*}
 \det \,({\bf{J}})=J_{11}J_{22}-J_{12}J_{21}
 =u^*v^*\frac{\partial g_1}{\partial u^*}\frac{\partial g_2}{\partial v^*} - u^*v^*\frac{\partial g_1}{\partial v^*} \frac{\partial g_2}{\partial u^*}.   
\end{align*}
By using implicit function theorem as used in \cite{S21}, we obtain
\begin{align*}
 \det \,({\bf{J}})=u^*v^*\frac{\partial g_1}{\partial v^*}\frac{\partial g_2}{\partial v^*}\left(\frac{dv^{*(g_2)}}{du^*}-\frac{dv^{*(g_1)}}{du^*} \right),
\end{align*}
where $\frac{dv^{*(g_2)}}{du^*}$ and $\frac{dv^{*(g_1)}}{du^*}$ are the slopes of the tangents of the  predator and prey nullclines at $E_2$ respectively.
\begin{itemize}
\item[(a)] 
Now we assume $p+m\geq 1$. The two possibilities for the sign of the Jacobian matrix at $E_2$ are:

\begin{align}\label{case1}
 \text{sign}(\bf{J}) &= 
     \begin{bmatrix}
   - & -\\
     +& -
     \end{bmatrix}     
\end{align}
or
\begin{align}\label{case2}
 \text{sign}(\bf{J}) &= 
     \begin{bmatrix}
   + & -\\
     +& -
     \end{bmatrix}.
\end{align}

If $b>c(1-p)u^{*p-2}v^{*m}$, then  $\det \,({\bf{J}})>0$  and  $\operatorname{tr} \,({\bf{J}})<0$. Thus in \eqref{case1}, $E_2$ is LAS by  Routh-Hurwitz criterion. In \eqref{case2}, the $\det \,({\bf{J}})>0$ since  
$\frac{dv^{*(g_2)}}{du^*}>\frac{dv^{*(g_1)}}{du^*}$. This is clearly seen in Fig. \ref{fig:1and2_intEquilibrium}[(a) and (b)]. Thus, $E_2$ is LAS if the  $\operatorname{tr} \,({\bf{J}})<0$.

\item[(b)] Assume $p+m<1$ and there exist two coexistence equilibrium points. The sign of the Jacobian matrix at $E^1_2$ is given by

\begin{align}\label{case3}
 \text{sign}(\bf{J}) &= 
     \begin{bmatrix}
   +& -\\
     +& -
     \end{bmatrix}     
\end{align}
and $E^2_2$ is
\begin{align}\label{case4}
 \text{sign}(\bf{J}) &= 
     \begin{bmatrix}
   -& -\\
     +& -
     \end{bmatrix}
\end{align}
or
\begin{align}\label{case5}
 \text{sign}(\bf{J}) &= 
     \begin{bmatrix}
   +& -\\
     +& -
     \end{bmatrix}
\end{align}

In \eqref{case3}, $b<c(1-p)u^{*p-2}v^{*m}$ and $\frac{dv^{*(g_2)}}{du^*}<\frac{dv^{*(g_1)}}{du^*}$, hence $E^1_2$ is a saddle point since the $\det \,({\bf{J}})<0$. This is evident in Fig. \ref{fig:1and2_intEquilibrium}(c). In \eqref{case4}, $b>c(1-p)u^{*p-2}v^{*m}$, thus $E^2_2$ is LAS since  $\det \,({\bf{J}})>0$  and  $\operatorname{tr} \,({\bf{J}})<0$. In \eqref{case5}, when $b<c(1-p)u^{*p-2}v^{*m}$ then  $\det \,({\bf{J}})>0$ if $\frac{dv^{*(g_2)}}{du^*}>\frac{dv^{*(g_1)}}{du^*}$. Therefore, $E^2_2$ is LAS if $\operatorname{tr} \,({\bf{J}})<0$.
\end{itemize}
\end{proof}
\begin{remark}\label{remark: on singularity of E0 and E1}
Since $0<p,m<1$, there is singularity in the Jacobian at $E_0$ and $E_1$. Hence we cannot analyze the stability of $E_0$ and $E_1$ by linearizing the  model \eqref{EquationMain}. We assume $E_0$ is a saddle point and we shall describe the behavior of model \eqref{EquationMain} by using the manifolds produced from $E_0$.  In the phase portrait, all trajectories below the stable manifold are attracted towards the stable coexistence equilibrium point and those above the stable manifold go the predator axis i.e.  the prey population goes extinct in finite time and consequently the extinction of predator population asymptotically. 
Any of the coexistence equilibria cannot be globally asymptotically stable (GAS) if LAS due to the singularity at $E_0$.
\end{remark}

\section{Bifurcation Analysis}\label{Sect:bifurcation_analysis}
\subsection{Local Bifurcation}
In this subsection, we investigate the qualitative changes in the dynamical behavior of model \eqref{EquationMain} under the effect of varying the strength of the fear of predator $k$. The conditions and restrictions for the occurrence of saddle-node and Hopf bifurcations are derived analytically and their classification is of co-dimension $1$ bifurcations. Additionally, we present numerically the two-parameter projection of the Hopf-bifurcation diagrams of the model \eqref{EquationMain}.

\subsubsection{Saddle-node bifurcation}
Saddle-node bifurcation occurs when shifting a parameter value causes two equilibria of contrasting stability to collide and mutually disappear, forming an instantaneous saddle-node at the point of their collision.
In the next theorem, we show that using the strength of fear as a bifurcation parameter, the model \eqref{EquationMain} satisfies the conditions for saddle-node bifurcation.

\begin{theorem}\label{thm:SN-k}
Model \eqref{EquationMain} admits a saddle-node bifurcation around $E_2$ at $k_s$ when the model parameter values satisfy the conditions $\frac{dv^{*(g_2)}}{du^*}=\frac{dv^{*(g_1)}}{du^*}$ and $\operatorname{tr} \,({\bf{J}})<0$.
\end{theorem}

\begin{proof}
In order to verify the conditions for the existence of saddle-node bifurcation, we employ Sotomayor's theorem \cite{Perko13} at $k=k_s$.  At $k=k_s$, we obtain that
 $\frac{dv^{*(g_2)}}{du^*}=\frac{dv^{*(g_1)}}{du^*}$ and $\operatorname{tr} \,({\bf{J}})<0$, which shows that the Jacobian $\,({\bf{J}})$ has a zero eigenvalue. Let $W$ and $Z$ be the eigenvectors corresponding to the zero eigenvalue of the matrix $\,{\bf{J}}$ and $\,{\bf{J}}^T$ respectively. Here, $W=(w_1,w_2)^T$ and $Z=(z_1,z_2)^T$, where $w_1=-\frac{J_{12}w_2}{J_{11}}$, $z_1=-\frac{J_{21}z_2}{J_{11}}$ and $w_2,z_2 \in \mathbb{R} \setminus \{0\}$.

Furthermore, let $H=(F,G)^T$ and $\tilde{M}=(u^*,v^*)^T$, where $F, G$ are defined in \eqref{EquationMain}.
Thus
\begin{align*}
Z^T H_{k}(\tilde{M},k_s) = (z_1,z_2)\left(-\dfrac{au^*v^*}{(1+k_sv^*)^2},0\right)^T=-\dfrac{au^*v^*}{(1+k_sv^*)^2}z_1 \neq 0,
\end{align*}
and
\begin{align*}
Z^T \left[D^2 H(\tilde{M},k_s)(W,W)\right] \neq 0.
\end{align*}
Therefore model \eqref{EquationMain} admits a saddle-node bifurcation when $k=k_s$.
\end{proof}

\begin{remark}
Additionally, the model \eqref{EquationMain} undergoes saddle-node bifurcation around $E_2$ with respect to the following parameters, $d,b,a,c$ and $e$. See Fig. \ref{fig:SN-Appendix} in the appendix for numerical verification.
\end{remark}

\begin{figure}[!htb]
\begin{center}
\subfigure[]{
    \includegraphics[scale=.263]{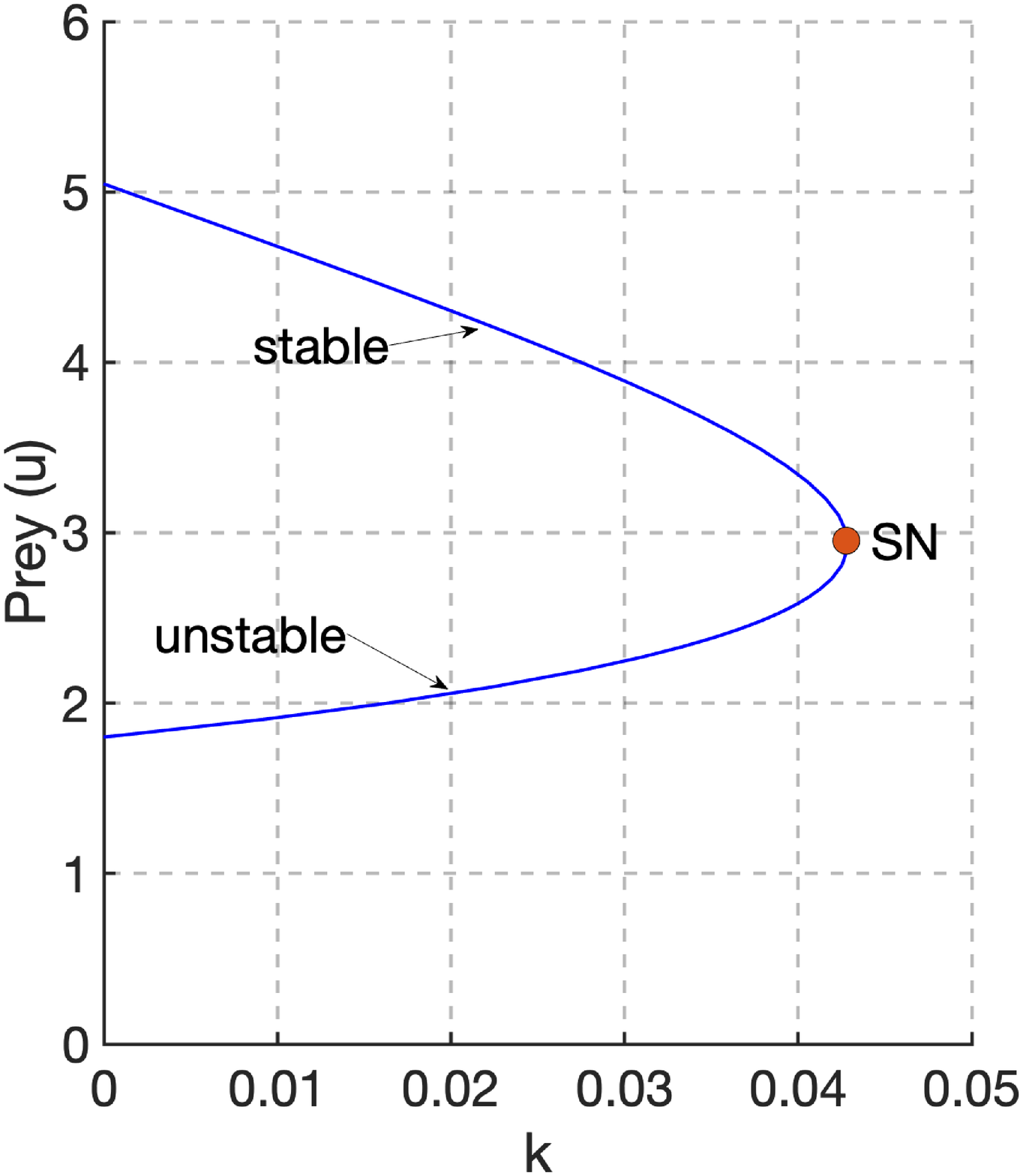}}
\subfigure[]{    
    \includegraphics[scale=.263]{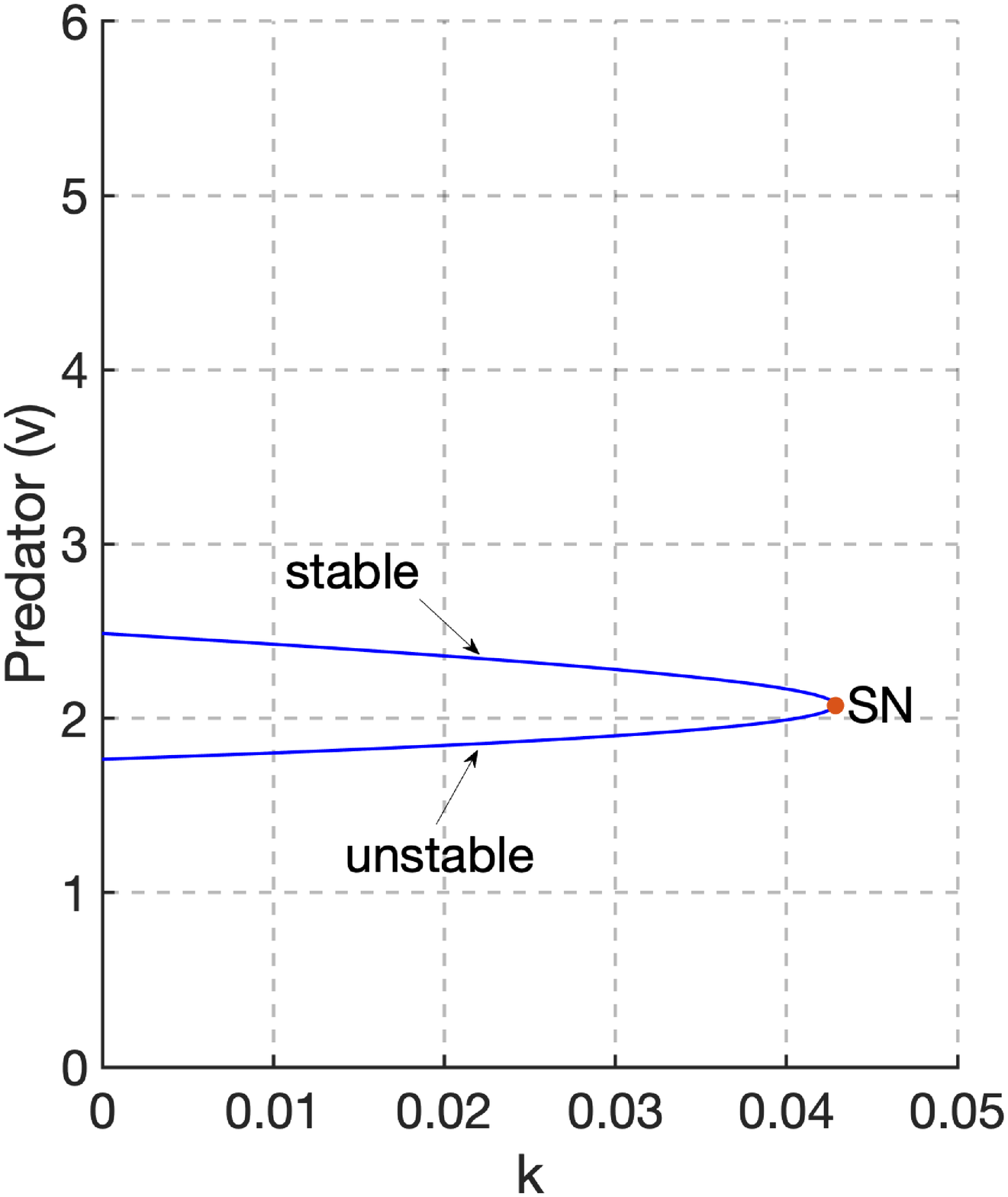}}
\end{center}
 \caption{Figures illustrating saddle-node bifurcation of the model \eqref{EquationMain}  at $k=k_s=0.042859$. Here, $m=0.4,~a=2.5,~b=0.3,~p=0.2,~d=2,~c=2.5,~e=2.5$. (SN: Saddle-node point)  }
      \label{fig:SN_Main}
\end{figure}


\subsubsection{Hopf-bifurcation}
Similar to a saddle-node bifurcation, a Hopf-bifurcation describes a local change in the stability of an interior equilibrium point due to an alteration of a parameter. However, for a Hopf-bifurcation, varying the bifurcation parameter does not annihilate or create new equilibrium points. Rather, at the point where system stability shifts (i.e. Hopf point) – a stable or unstable periodic orbit develops.
The conditions for the existence of Hopf-bifurcation of the model \eqref{EquationMain} is derived in the theorem below.

\begin{theorem}\label{bifurcationMain}
Model \eqref{EquationMain} experiences Hopf-bifurcation around the coexistence equilibrium point $E_2$ at $k=k_h$, where
\begin{align*}
k_h=\frac{1}{v^*}\left[\frac{a}{2bu^*+cpu^{*p-1}v^{*m}+d(1-m)}-1\right],
\end{align*}
 when the following conditions are satisfied:\\
 $S(k)=0 $, $M(k)>0$, and
 $\dfrac{d}{dk}\left. \text{Re} [\eta_i (k) \right]|_{k=k_h}\neq 0$ for $ i=1,2$.
\end{theorem}

\begin{proof}
Using the strength of fear as a bifurcation parameter, consider the Jacobian matrix \eqref{Jacobian} around the coexistence equilibrium $E_2$. The characteristic equation at $E_2$ is given by
\begin{align}\label{character_bifur_main}
\eta^2 - S(k)\eta + M(k) = 0,
\end{align}
where
$S = \operatorname{tr} \,({\bf{J}}) = J_{11} + J_{22}$ and
$M =  \det \,({\bf{J}}) =  J_{11}  J_{22}- J_{12} J_{21}.$
The zeros of the equation \eqref{character_bifur_main} are
\begin{align}\label{cond:HB}
\eta_{1,2}=\xi (k) \pm i \mu (k).
\end{align}
At $k=k_h$, $S(k)=0$ implies
\begin{align*}
\dfrac{a}{1+k_hv^*}-2bu^*-pcu^{*p-1}v^{*m}-d(1-m)&=0.
\end{align*}
The characteristic equation \eqref{character_bifur_main} becomes
 \begin{align}\label{charac_reduced_bifur_main}
 \eta^2 + M(k)=0,
 \end{align}
at $k=k_h$.  Solving for the zeros of equation \eqref{charac_reduced_bifur_main} yields $\eta_{1,2}=\pm i\sqrt{M}$. Thus, a pair of purely imaginary eigenvalues. Furthermore, we substantiate  the transversality condition. For any $k$ in the neighborhood of $k_h$ in \eqref{cond:HB}, let $\xi(k)=\text{Re} \left[\eta_i (k) \right]=\frac{1}{2}S(k)$ and $\mu(k)=\sqrt{M(k)-\frac{[S(k)]^2}{4}}$. Thus, 
\begin{align*}
\dfrac{d}{dk}\left. \text{Re} [\eta_i (k) \right]|_{k=k_h}=\dfrac{1}{2}\dfrac{d}{dk}S(k)|_{k=k_h}=-\dfrac{1}{2}\dfrac{av^*}{(1+k_hv^*)^2}\neq 0.
\end{align*}
 Therefore, by the Hopf-bifurcation Theorem \cite{Murray93}, the model \eqref{EquationMain} experiences a Hopf-bifurcation around $E_2$ at $k=k_h$.
\end{proof}


\subsubsection{Direction of Hopf-bifurcation}
We investigate the stability and direction of the periodic cycles  emitted via Hopf-bifurcation around the coexistence  equilibrium point  by computing the first Lyapunov coefficient \cite{Perko13}. We first translate the coexistence equilibrium $E_2$ of the model \eqref{EquationMain} to the origin by using the transformation $x=u-u^*$ and $y=v-v^*$. Now, the model \eqref{EquationMain} becomes
\begin{equation*}
\begin{cases}
\dfrac{dx }{dt} &=\dfrac{a (x+u^*)}{1+k (y+v^*)}  - b  (x+u^*)^2 -  c  (x+u^*)^p (y+v^*)^m ,\\
\dfrac{dy }{dt} &=-d  (y+v^*) + e  (x+u^*)^p (y+v^*)^m.
\end{cases}
\end{equation*}
Applying Taylor series expansion at $(x,y)=(0,0)$ up to third order, we obtain the following planar analytic model
\begin{equation}\label{Lyapunov_eqn}
\begin{cases}
 \dot{x} &=a_{10}x +  a_{01}y +  a_{20}x^2+  a_{11}xy+  a_{02}y^2 +  a_{30}x^3+  a_{21}x^2 y \\
 &+  a_{12}x y^2+  a_{03}y^3,  \\
\dot{y} &= b_{10}x+  b_{01}y+  b_{20}x^2+  b_{11}xy+  b_{02}y^2 +  b_{30}x^3+  b_{21}x^2 y\\
&+  b_{12}x y^2+  b_{03}y^3,
\end{cases}
\end{equation}
where
\begin{align*}
a_{10}=&\dfrac{a}{1+k_hv^*}-2bu^*-cpu^{*p-1}v^{*m},\\
a_{01}=&-\dfrac{k_hau^*}{(1+k_hv^*)^2}-cmu^{*p}v^{*m-1},\\
a_{20}=&- b-\frac{c}{2} (p-1) p u^{*p-2} v^{*m},\\
a_{11}=&-\frac{a k_h}{(1+k_h v^*)^2}-c m pu^{*p-1}  v^{*m-1},\\
a_{02}=&\frac{a k_h^2 u^*}{(1+k_h v^*)^3}-\frac{c}{2} (m-1) mu^{*p}  v^{*m-2},\\
a_{30}=& -\frac{c}{6} (p-2) (p-1) p u^{*p-3}  v^{*m},\\
a_{21}=& -\frac{c}{2} m (p-1) p  u^{*p-2} v^{*m-1},\\
a_{12}=&\frac{a k_h^2}{(1+k_h v^*)^3}-\frac{c}{2} (m-1) m p u^{*p-1} v^{*m-2},\\
a_{03}=& -\frac{a k_h^3 u^*}{(1+k_h v^*)^4}-\frac{c}{6} (m-2) (m-1) m u^{*p} v^{*m-3},\\
b_{10}=& epu^{*p-1}v^{*m},\\
b_{01}=& -d+emu^{*p}v^{*m-1},\\
b_{20}=& \frac{e}{2} (p-1) p u^{*p-2} v^{*m},\\
b_{11}=& e m p u^{*p-1} v^{*m-1},\\
b_{02}=& \frac{e}{2} (m-1) m u^{*p} v^{*m-2},\\
b_{30}=& \frac{e}{6} (p-2) (p-1) p u^{*p-3}  v^{*m},\\
b_{21}=& \frac{e}{2} m (p-1) p  u^{*p-2} v^{*m-1},\\
b_{12}=& \frac{e}{2} (m-1) m p  u^{*p-1} v^{*m-2},\\
b_{03}=& \frac{e}{6} (m-2) (m-1) m u^{*p} v^{*m-3}.\\
\end{align*}
Since $a_{10}, a_{01}, b_{10}$ and $b_{01}$ are the components of the Jacobian matrix $\bf{J}$ evaluated at the coexistence equilibrium point $E_2$ at $k=k_h$, then $S=a_{10}+b_{01}=0$ and $M=a_{10}b_{01}-a_{01}b_{10}>0$.\\
The first Lyapunov coefficient $L$ \cite{Perko13} is computed by the formula
\begin{align*}
L =\frac{-3\pi}{2a_{01}M^{3/2}} \{ [ a_{10}b_{10}(a^2_{11}+a_{11}b_{02}+a_{02}b_{11}) + a_{10}a_{01}(b^2_{11}+a_{20}b_{11}+a_{11}b_{02}) \\
+b^2_{10}(a_{11}a_{02}+2a_{02}b_{02}) - 2a_{10}b_{10}(b^2_{02}-a_{20}a_{02}) - 2a_{10}a_{01}(a^2_{20}-b_{20}b_{02}) \\
-a^2_{01}(2a_{20}b_{20}+b_{11}b_{20})+(a_{01}b_{10}-2a^2_{10})(b_{11}b_{02}-a_{11}a_{20})] \\
-(a^2_{10}+a_{01}b_{10})\left[ 3(b_{10}b_{03}-a_{01}a_{30})+2a_{10}(a_{21}+b_{12})+(b_{10}a_{12}-a_{01}b_{21}) \right] \}.
\end{align*}
Now, if $L<0$, then the Hopf-bifurcation is \emph{supercritical} and \emph{subcritical} if $L>0$.

\begin{figure}[!htb]
\begin{center}
\subfigure[]{
    \includegraphics[scale=.26]{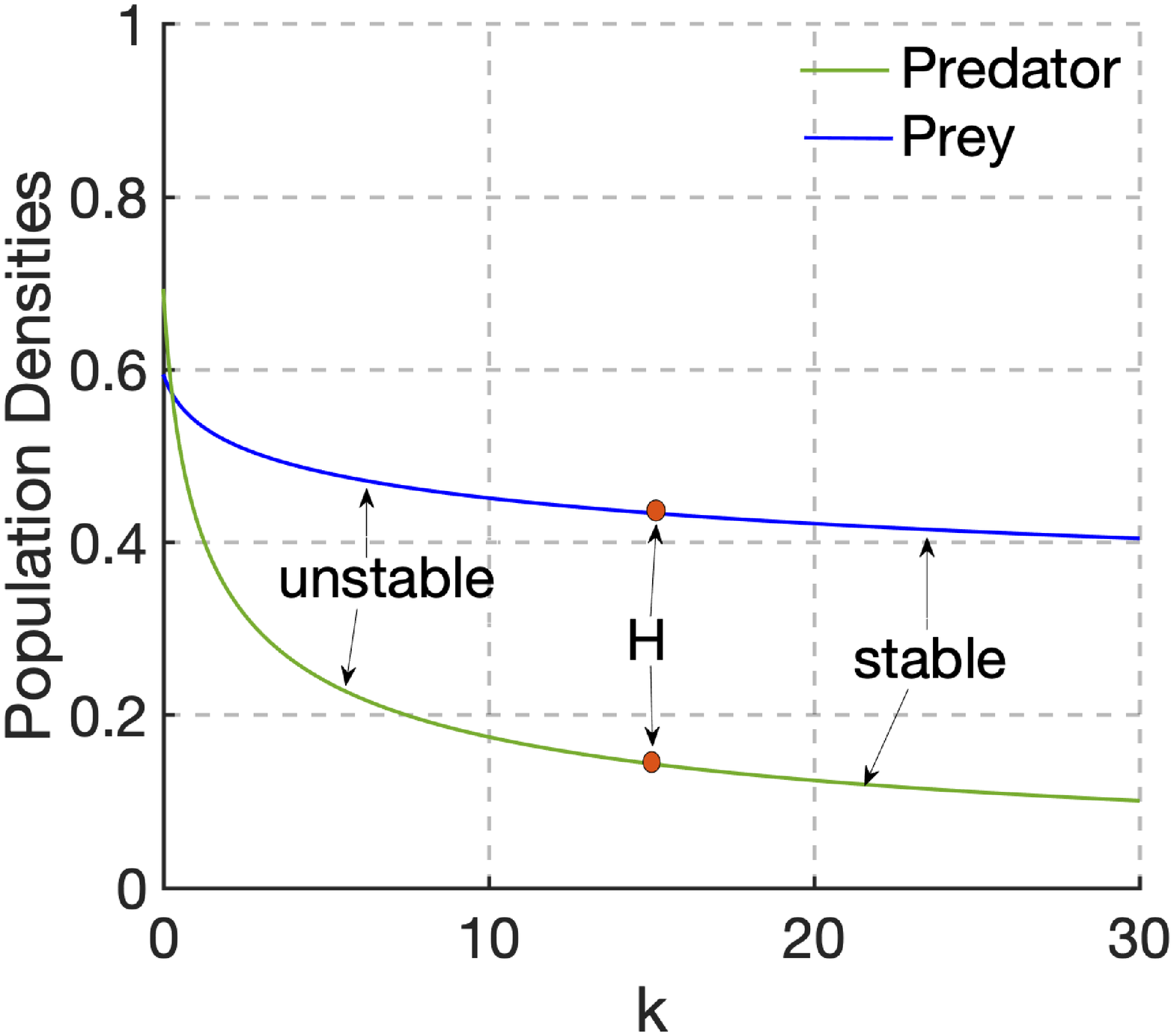}}
\subfigure[]{    
    \includegraphics[scale=.26]{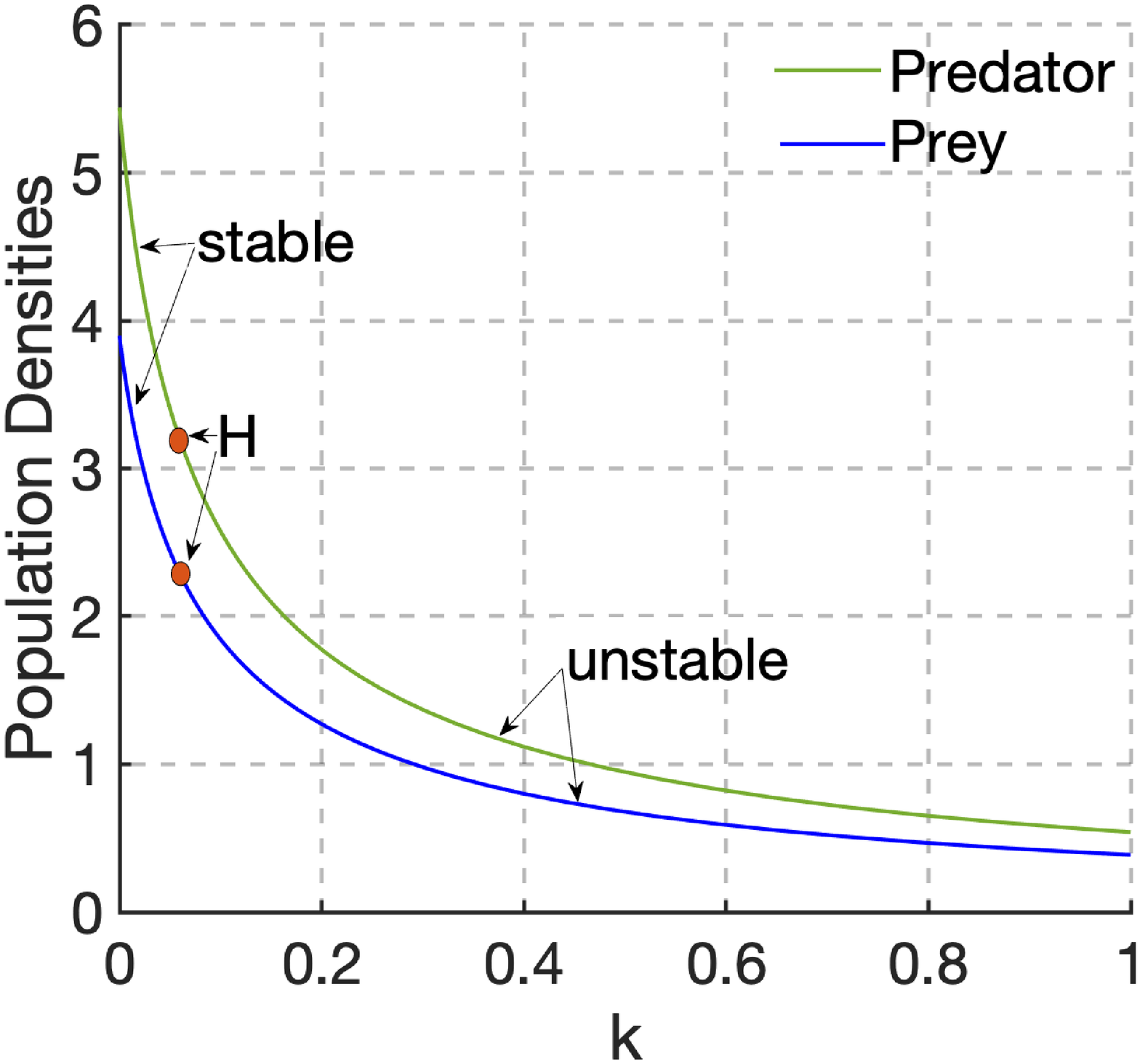}}
\end{center}
 \caption{Bifurcation diagrams of the model \eqref{EquationMain} illustrating change in stability with $k$ as the bifurcating parameter. 
 (a) Hopf point at $k=k_h=15.093353$ with 
first Lyapunov coefficient$=1.83219e^{-01}$, here $m=0.9,~a=2.5,~k=10,~b=0.3,~p=0.5,~d=2,~c=2.5,~e=2.5$ (b) Hopf point at $k=k_h=0.061382$  with 
first Lyapunov coefficient$=1.76388e^{-02}$, here $m=0.6,~a=2,~b=0.2,~p=0.4,~d=0.9,~c=1,~e=0.8.$ (H: Hopf point)}
      \label{fig:Hopf_Main}
\end{figure}

\begin{remark}\label{remark:stability}
From the numerical simulations in Fig. \ref{fig:Hopf_Main}, we show that the fear of predator has an effect in altering the stability of the coexistence equilibrium solution $E_2(u^*,v^*)$ of the model \eqref{EquationMain} via Hopf-bifurcation. In Fig. \ref{fig:Hopf_Main}(a), the coexistence equilibrium solution changes from an unstable zone to a stable zone around $E_2(0.43392,0.14327)$ as the strength of fear of predator $k$ crosses the Hopf point at $k_h=15.093353$. We used MATCONT \cite{Matcont} to generate the bifurcation diagrams and obtained $L=1.83219e^{-01}>0$, hence subcritical Hopf-bifurcation. Furthermore, in Fig. \ref{fig:Hopf_Main}(b), the coexistence equilibrium solution changes from a stable zone to an unstable zone around $E_2(2.26530,3.16304)$ as the parameter $k$ crosses the Hopf point at $k_h=0.061382$ with $L=1.76388e^{-02}>0$, hence subcritical Hopf-bifurcation. In summary, we conclude here that with appropriate parameters the effect of fear of predator can have a stabilizing and destabilizing effect on the coexistence equilibrium solutions of the model \eqref{EquationMain}.
\end{remark}

\begin{figure}[!htb]
\begin{center}
\subfigure[]{
    \includegraphics[scale=.170]{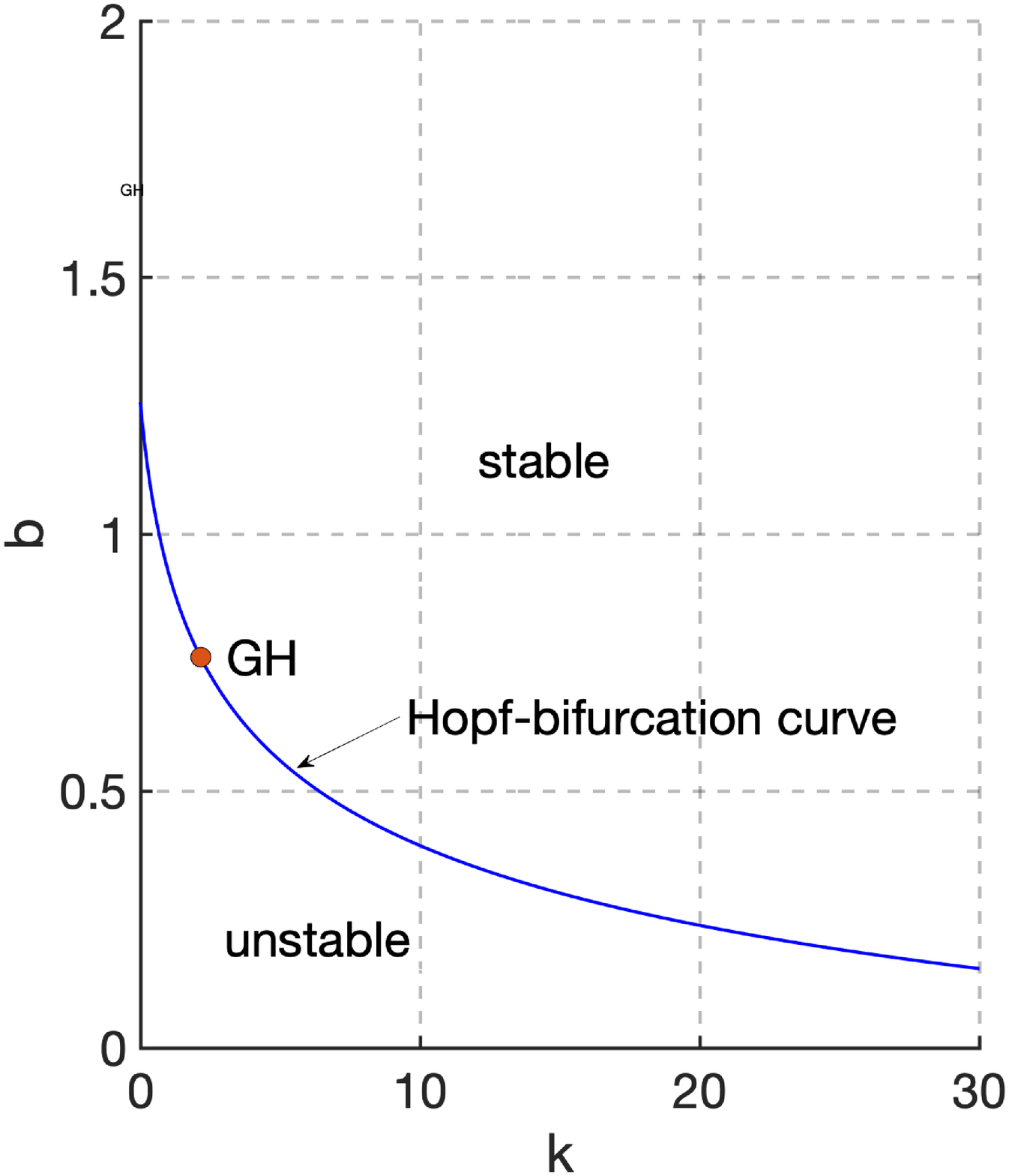}}
\subfigure[]{    
    \includegraphics[scale=.170]{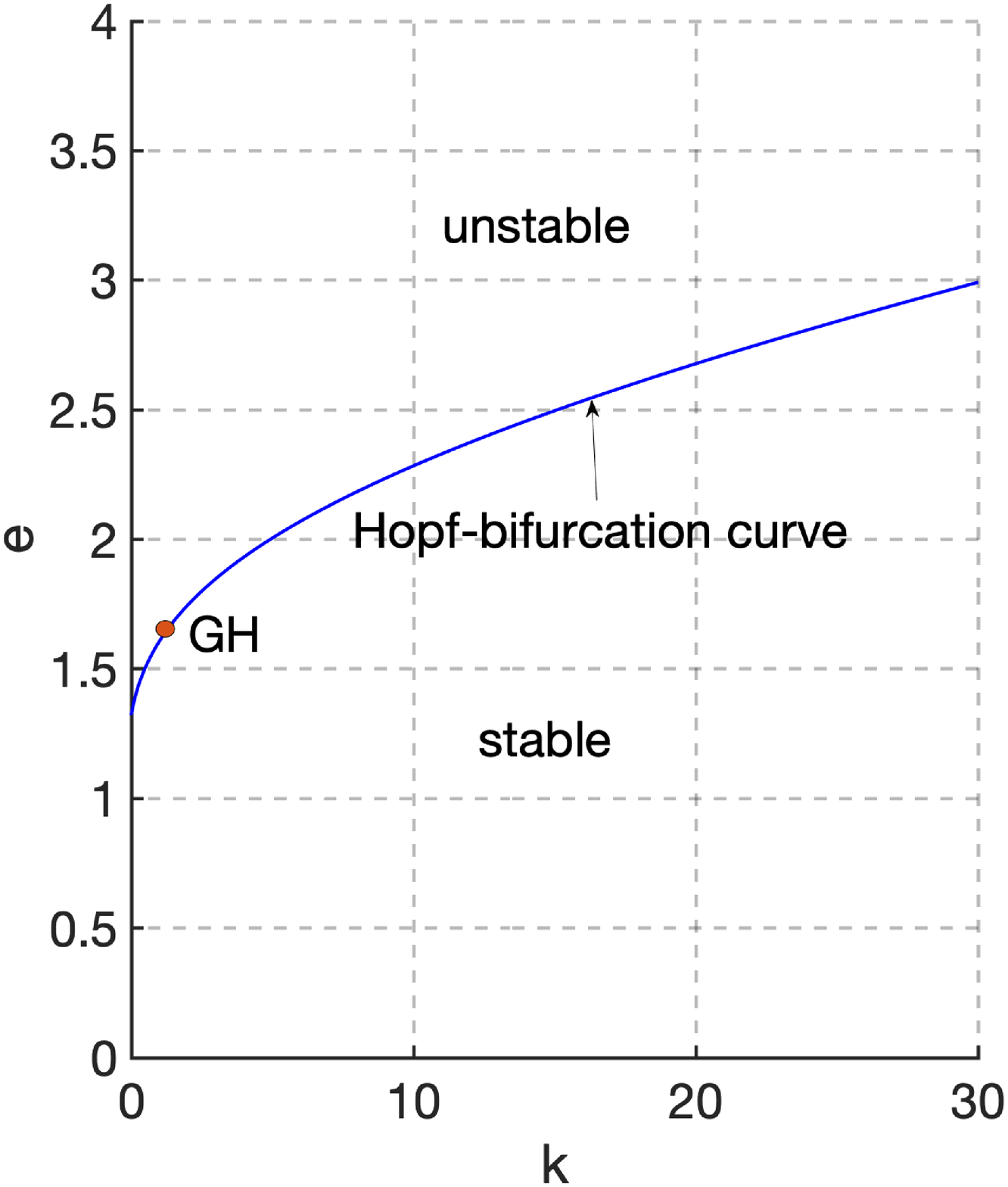}}
\subfigure[]{    
    \includegraphics[scale=.170]{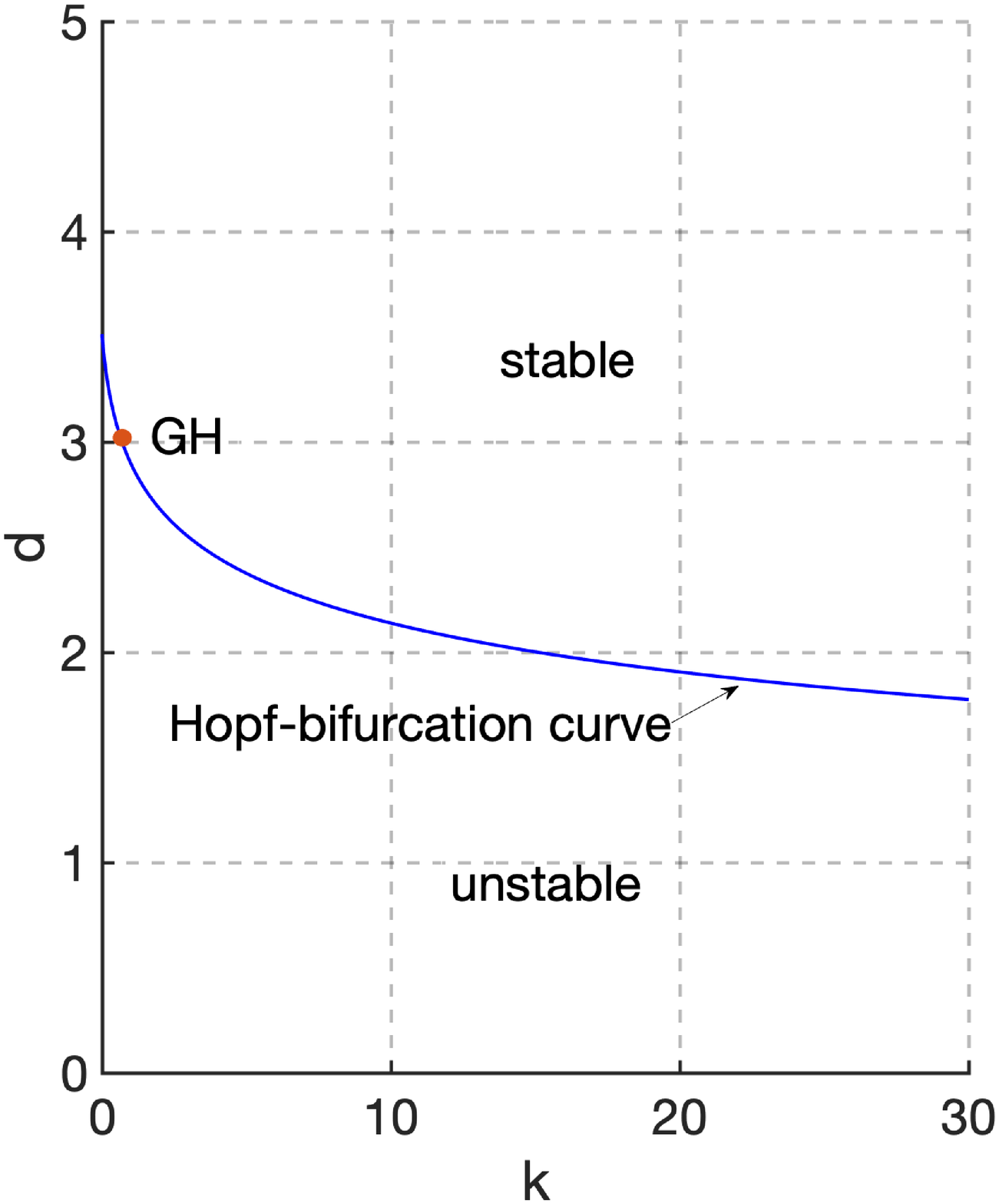}}    
\subfigure[]{    
    \includegraphics[scale=.170]{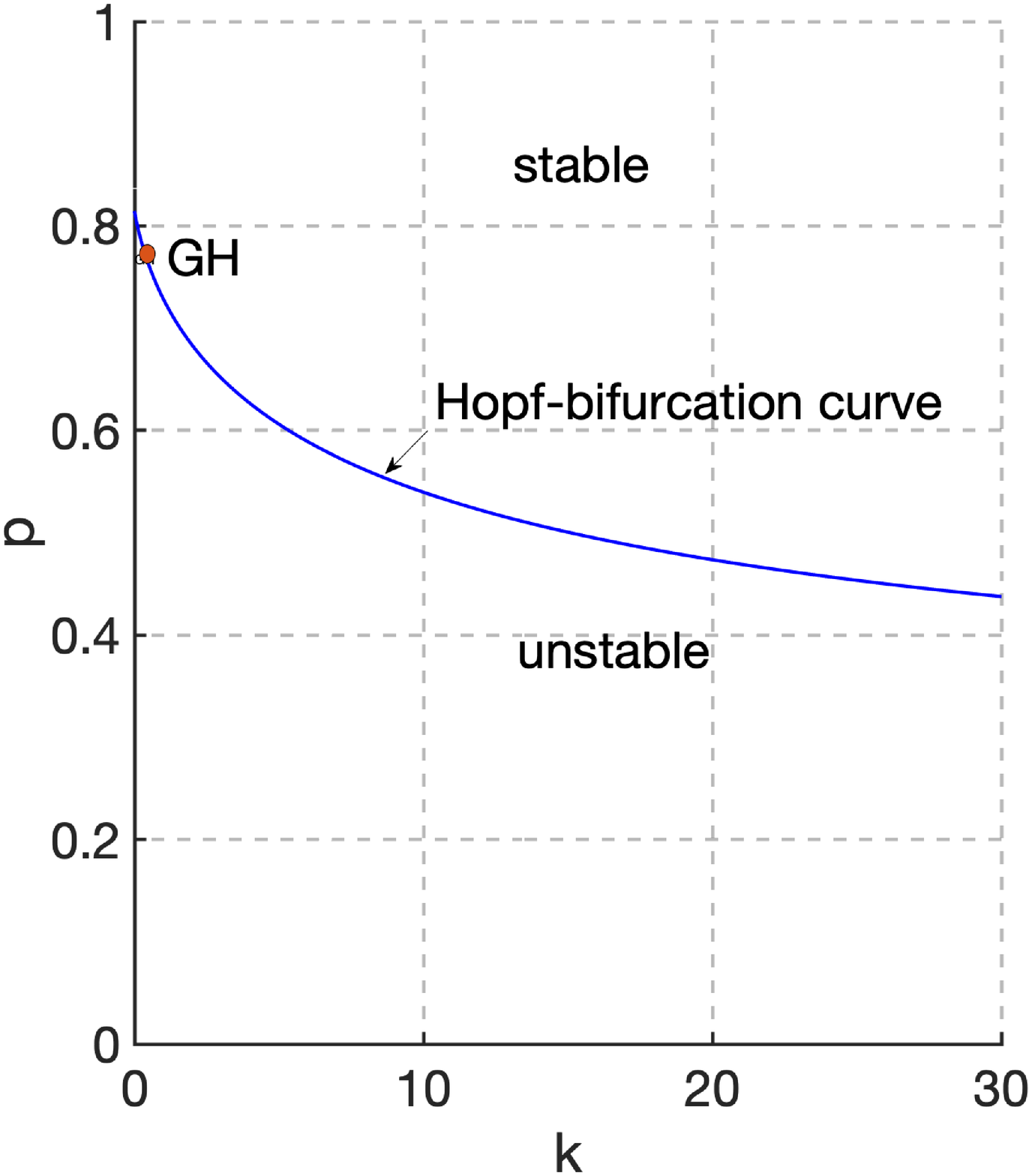}}   
\subfigure[]{    
    \includegraphics[scale=.170]{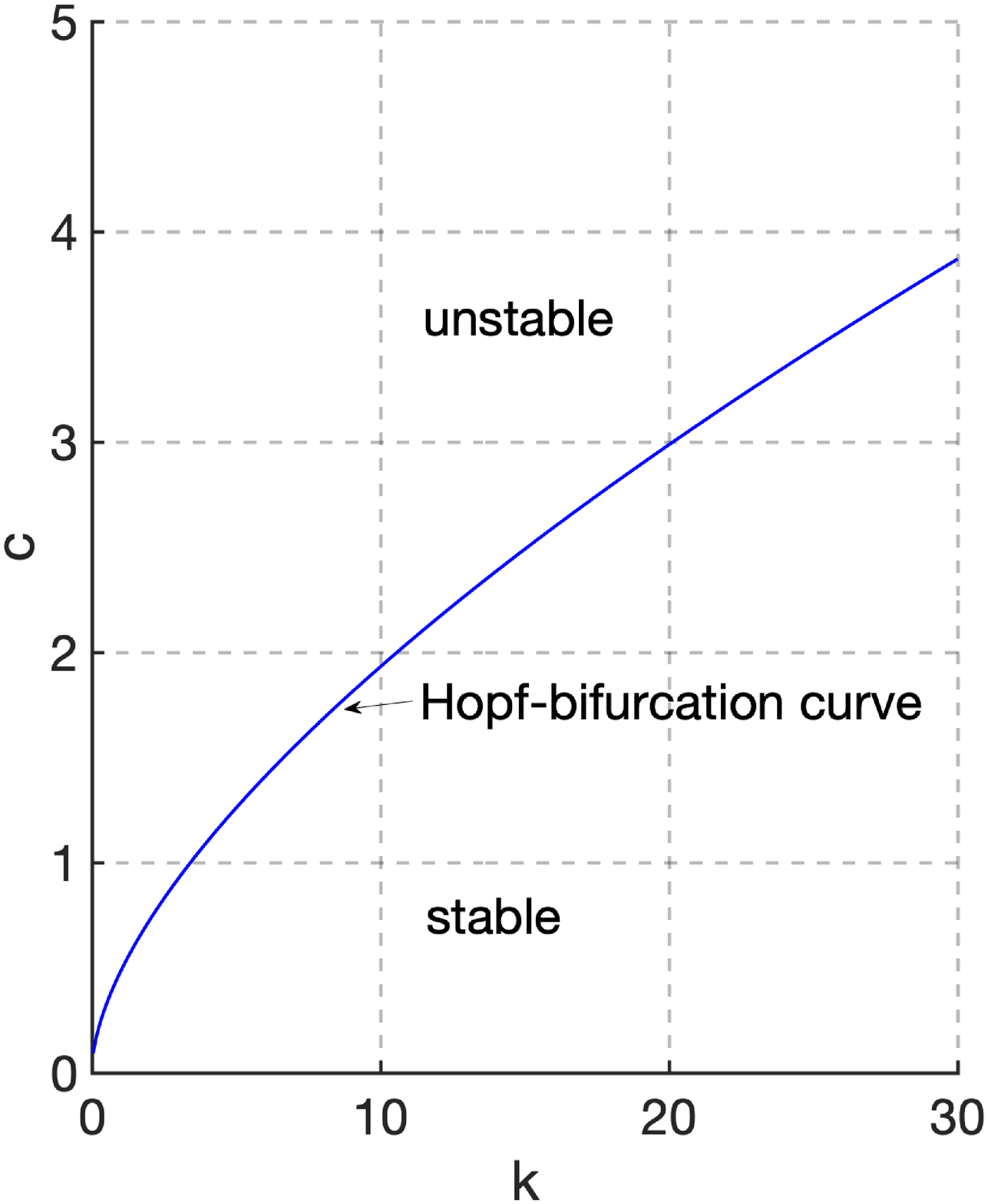}}
\subfigure[]{    
    \includegraphics[scale=.170]{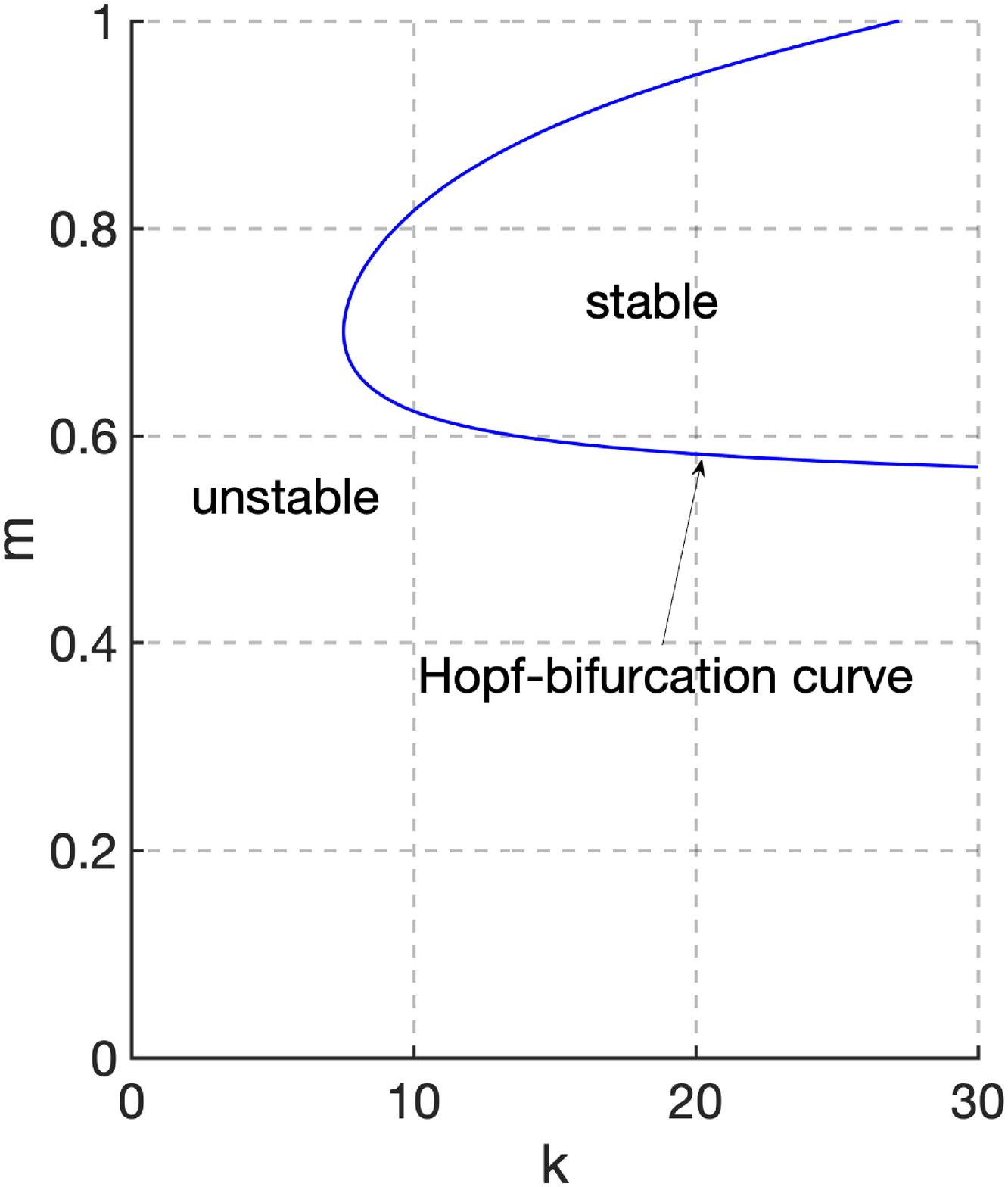}}     
\end{center}
\caption{Two-parameter bifurcation diagrams of the model \eqref{EquationMain}. Here, we use parameters from Fig. \ref{fig:Hopf_Main}(a). (GH: Generalized Hopf point)}
\label{fig:2param_bifurcation}
\end{figure}
\subsubsection{Generalized Hopf-bifurcation}
We note here that from the  two-dimensional projection of the Hopf-bifurcation diagrams of the model \eqref{EquationMain} in Fig \ref{fig:2param_bifurcation},  a generalized Hopf-bifurcation or Bautin bifurcation is observed. The generalized Hopf-bifurcation is a local bifurcation of co-dimension 2 and this happens when the first Lyapunov coefficient is zero, and the coexistence equilibrium point has a pair of purely imaginary eigenvalues. The generalized Hopf-bifurcation point separates branches of subcritical and supercritical Hopf-bifurcation in the parameter plane. 

Now, we present a conjecture that pertains to generalized Hopf-bifurcation.
\begin{conjecture}[Existence of Generalized Hopf-bifurcation]\label{conjecture:Bautin}
Assume that the model \eqref{EquationMain}  admits a Hopf-bifurcation as in Theorem \ref{bifurcationMain} for a given set of parameters. If the first Lyapunov coefficient becomes zero and the coexistence equilibrium point has a pair of purely imaginary eigenvalues, then the model \eqref{EquationMain} undergoes a Bautin or generalized Hopf-bifurcation.
\end{conjecture}
\subsection{Global Bifurcation}
 Now, there exist a unique value of parameter $k$ for which   $W^s(E_0)$ and $W^u(E_0)$ coincide, and that implies existence of a homoclinic loop in model \eqref{EquationMain}. In particular, we state a conjecture concerning the effect of fear of predators on the global dynamics of model \eqref{EquationMain}.
\begin{conjecture}[Existence of Homoclinic Bifurcation]\label{conjecture:Homoclinic_k}
Consider the model \eqref{EquationMain} where all  parameters are fixed except $k\geq 0$. There exists $k^*>0$ such that a homoclinic loop occurs when $k=k^*$.
\end{conjecture}
Next, we explain Conjecture \ref{conjecture:Homoclinic_k} via numerical simulations in Fig. \ref{fig:HomoclinicMain_Curve}.
\begin{figure}[!htb]
\begin{center}
\subfigure[]{
    \includegraphics[scale=.69]{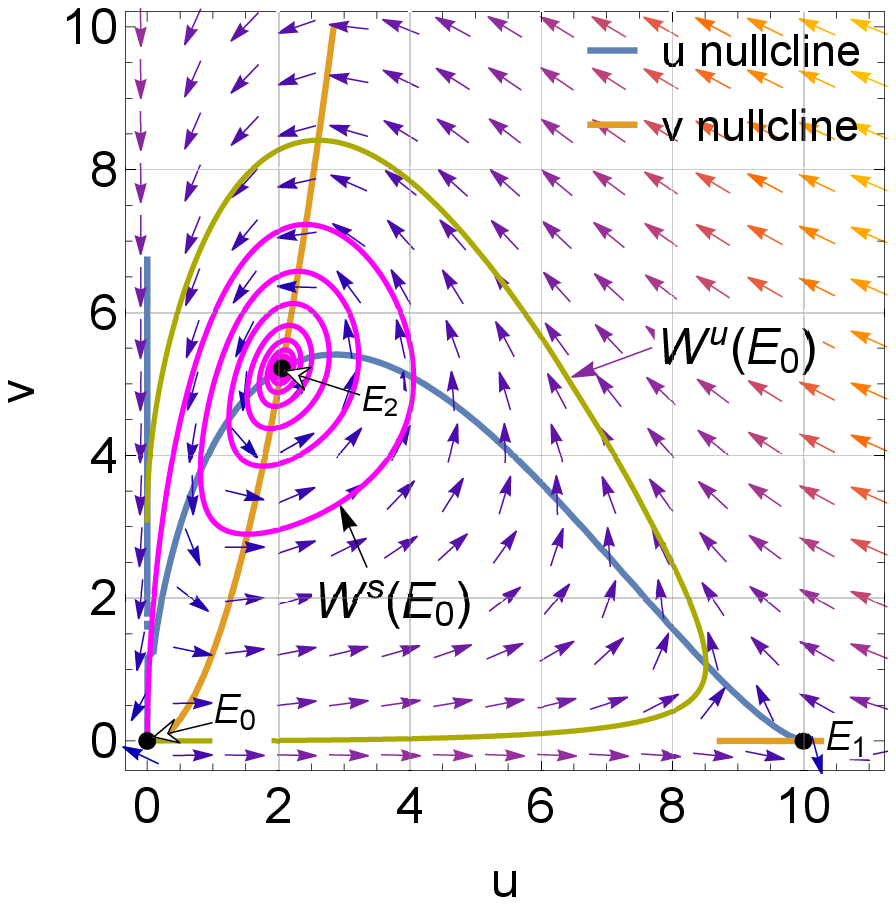}}
\subfigure[]{    
    \includegraphics[scale=.69]{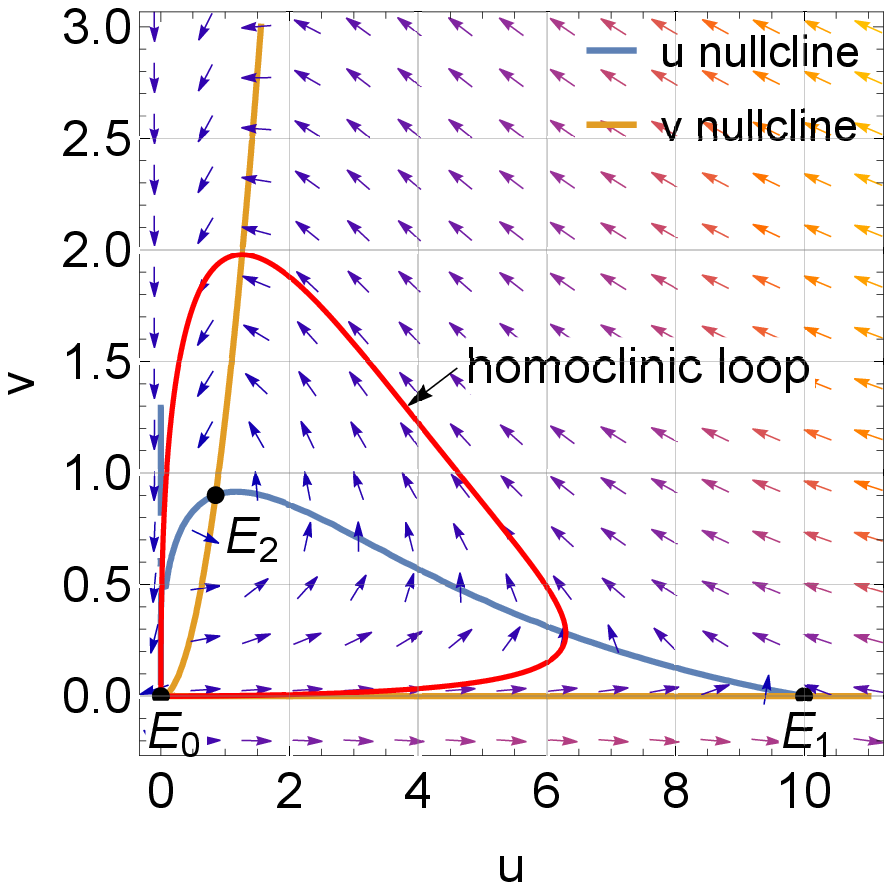}}
\end{center}
\caption{Phase plane portraits of the model \eqref{EquationMain}  (a) Stable ($W^s(E_0)$) and unstable ($W^u(E_0)$) manifold with $k=0$ where $W^s(E_0)$ is below $W^u(E_0)$, here $E_2$ is unstable   (b) $E_2$ is stable and a homoclinic loop is formed when $W^s(E_0)=W^u(E_0)$ for $1.548<k<1.563$. Here the parameters used are $m=0.7,~a=3,~b=0.3,~p=0.6,~d=0.75,~c=1$ and $e=0.8.$ }
\label{fig:HomoclinicMain_Curve}
\end{figure}

\section{Fear-Driven Finite Time Extinction (FDFTE)}\label{Sec:FDFTE}
We seek to investigate the effect of introducing fear of predator in the predator-prey model -- that is, is it possible for the fear effect to drive a stable coexistence equilibrium point to extinction in finite time?

Thus, we state our result concerning finite time extinction driven by fear of predators.

\begin{theorem}[FDFTE] \label{thm:FDFTE}
Consider the predator-prey model given by \eqref{EquationMain}, and a certain parameter set, and certain initial data $(u(0),v(0))$ that converges uniformly to a stable coexistence equilibrium point $(u^*,v^*)$ for $k=0$. Then there exists $k>0$ such that all trajectories initiating from the same initial data $(u(0),v(0))$ will lead to finite time extinction of $u$, followed by $v$ going extinct asymptotically.
\end{theorem}

\begin{proof}
We argue by contradiction. We begin by assuming not. Thus for a certain parameter set, with $k=0$ and certain initial data 
$(u^{*}(0),v^{*}(0))$ that converges uniformly to a stable coexistence equilibrium point $(u^*,v^*)$, there exists a $k > 0$ s.t for trajectories emanating from the same initial data, and parameters, we will have
 \begin{equation}
 u \geq u^{*}(0)e^{-T^{*}} > 0,  
 \end{equation}
on $[0,T^{*}] ,\  \forall T^{*}$. Now,
$$\dfrac{dv}{dt}\geq -dv. $$
This implies,
\begin{align}
\label{eq:1v}
v(t)\geq v(0)e^{-dt}.
\end{align}
Note, via \eqref{eq:1v}, the upper  bound on $u$, and positivity of solutions, we have
\begin{eqnarray}
 \dfrac{du }{dt} &=& \dfrac{au}{1+kv} -  u^{2} - cu^{p}v^{m}, \nonumber \\
&\leq& \frac{a u}{kv}     - cu^{p}v^m  \nonumber \\
&\leq&  \frac{a^{2}}{kv}   - cu^{p}v^m  \nonumber \\
&\leq&   \frac{a^{2}e^{dt}}{kv(0)} - c (v(0))^m e^{-mdt} u^{p} \nonumber \\
\end{eqnarray}
 However, we see that the solution $u$ to,
 \begin{equation}
  \dfrac{du }{dt} = \frac{a^{2}e^{dt}}{kv(0)} - c (v(0))^m e^{-mdt} u^{p}
 \end{equation}
will go extinct in finite time if 
 \begin{equation}
\left(  \frac{a^{2}e^{(m+1)dt}}{k c (v(0))^{(m+1)}} \right)^{\frac{1}{p}}<  u_{0}.
 \end{equation}
Thus, given an initial data $(u^{*}(0),v^{*}(0))$, and $T^{*} > 0$ we can choose $k>>1$, s.t.
 \begin{equation}
\left(  \frac{a^{2}e^{(m+1)dt}}{k c (v^{*}(0))^{(m+1)}} \right)^{\frac{1}{p}}<  u^{*}_{0}, \ \ \forall t \in [0, T^{*}],
 \end{equation}
 so that we obtain,
  \begin{equation}
0 \leq  u \leq u^{*}(0)e^{-T^{*}}, 
 \end{equation}
on $[0,T^{*}]$, with $u$ being driven to extinction in finite time, from which the asymptotic extinction of $v$ follows.  Since  $T^{*}$ is arbitrary, we have derived a contradiction, and so the theorem is proved.
\end{proof}

\begin{remark}
We see from Theorem \ref{thm:FDFTE} that $k$ chosen sufficiently large will lead to prey extinction in finite time; however in practice, $k$ need not be large, see Fig. \ref{fig:FDFTE}. Thus, a necessary condition on the size of $k$ to cause finite time prey extinction remains unproven. 
\end{remark}
We elucidate FDFTE using the  numerical simulation in Fig. \ref{fig:FDFTE}
\begin{figure}[hbt!]
\begin{center}
    \includegraphics[scale=.6]{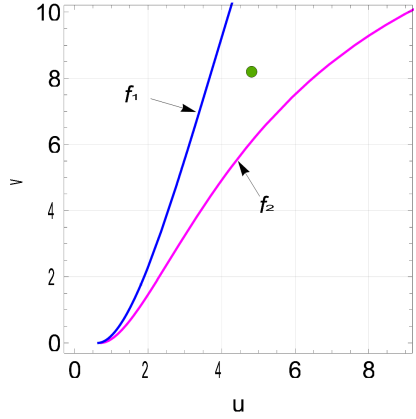}
\end{center}
\caption{Diagram demonstrating Theorem \ref{thm:FDFTE}. Herein, $f_1$ is the stable manifold of $E^1_2$ when $k=0$ and $f_2$ is the stable manifold of $E^1_2$  when $k=0.03$. The solid green circle represents an initial data at $(u(0),v(0))=(4.8,8.3)$.  Other parameter sets are given in the caption in Fig. \ref{fig:SN_Main}. }
\label{fig:FDFTE}
\end{figure}
%

\section{Impact of Effort Harvesting of Predators}\label{Sect:Harvesting}
\noindent
In this section, we consider the impact of an external effort dedicated to harvesting of predators in the model \eqref{EquationMain}. A natural question that arises is: how do the external effort of predator harvesting affect the FTE dynamic of the prey population species? Here, the harvesting is proportional to the density of harvested predator population species.  The corresponding differential equations can be represented as:
\begin{equation}\label{eq:HarvestingEqn}
\begin{cases}
\dfrac{du }{dt} &=\dfrac{a u}{1+k v}  - b u^2 -  c u^p v^m,   \qquad\; u(0)\geq 0,\\
\dfrac{dv }{dt} &=-vd-qv^r + e u^p v^m. \qquad \qquad v(0)\geq 0.
\end{cases}
\end{equation}
Let  the parameter $q>0$ represents the external effort dedicated to predator harvesting. Here $0<r\leq 1$.  We recover the model \eqref{EquationMain} when $q=0$. Nonnegativity of solutions of  model \eqref{eq:HarvestingEqn} follows from Lemma \ref{lemma:Nonnegativity}. The  model \eqref{eq:HarvestingEqn} contains a trivial equilibrium point $E_0(0,0)$, an axial equilibrium point, $E_1(a/b,0)$, and coexistence equilibrium point(s) $E_2$ (or $E^i_2$, for $i=1,2$).

\subsection{Dynamical Guidelines}
The linearized model \eqref{eq:HarvestingEqn} at any coexistence equilibrium point $(u^*,v^*)$ is given by the Jacobian Matrix $\bf{J^*}$
\begin{align}\label{Jacobian_Harvesting}
 \bf{J^*} &= 
     \begin{bmatrix}
     c_{11} &  c_{12}\\
      c_{21} &  c_{22}
     \end{bmatrix},
\end{align}
where\\
$c_{11}=J_{11}, c_{12}=J_{12},$ and $c_{21}=J_{21}$. Here $J_{11},J_{12}$ and $J_{21}$ are given by \eqref{Jacobian}. Now
\begin{align*}
c_{22}&= -d-qrv^{*r-1}+emu^{*p}v^{*m-1}\\
&=-d-qrv^{*r-1}+md+mqv^{*r-1}\\
&=-d(1-m)-q(r-m)v^{*r-1}. 
\end{align*}
%

\begin{theorem}\label{localstabilityHarvest}
Consider the model given by \eqref{eq:HarvestingEqn} and assume $r\geq m$.
\begin{itemize}
\item[(a)] There exists a unique  coexistence equilibrium point $E_2(u^*, v^*)$ which is LAS.
\item[(b)] There exist two coexistence equilibrium points such that $E^1_2(u_1^*, v_1^*)$ is a saddle and  $E^2_2(u_2^*, v_2^*)$ is LAS.
\end{itemize}
\end{theorem} 

\begin{proof}
The proof of Theorem \ref{localstabilityHarvest} is similar to proof in Theorem \ref{localstability} and therefore omitted.
\end{proof}

\begin{theorem}\label{thm:H-bifurcation_Harvest}
Model \eqref{eq:HarvestingEqn} experiences Hopf-bifurcation around the coexistence equilibrium point $E_2$ at $q=q_h$, where
\begin{align*}
q_h=\frac{1}{r}v^{*1-r}\left[emu^{*p}v^{*m-1}-bu+c(1-p)u^{*p-1}v^{*m}-d\right],
\end{align*}
 provided
 $S(q)=0 $, $M(q)>0$, and
 $\dfrac{d}{dq}\left. \text{Re} [\lambda (q) \right]|_{q=q_h}\neq 0$ for $ i=1,2$.
\end{theorem}

\begin{example}\label{example:Hopf_q}
To validate Theorem \ref{thm:H-bifurcation_Harvest}, we consider the following parameter values $m=0.6,~a=3,~k=0.08,~b=0.2,~p=0.5,~d=1,~c=2,~e=1.1,~q=1,~r=1$ (see Fig. \ref{fig:Hopf_HarvestParam_q}(a)). Hopf-bifurcation is obtained at $q=q_h=0.27068$ around the coexistence equilibrium point $E_2(2.54542,2.24175)$. Furthermore, at $q=q_h$, $S(q)=\operatorname{tr} \,({\bf{J^*}})=0$, $M(q)=\det \,({\bf{J^*}})=0.29264>0$ and $\frac{d}{dq}\left. \text{Re} [\lambda (q) \right]|_{q=q_h}=-1\neq 0$. Hence, all necessary and sufficient condition for Hopf-bifurcation to occur are satisfied.
\end{example}


\begin{figure}[hbt!]
\begin{center}
\subfigure[]{
    \includegraphics[scale=.261]{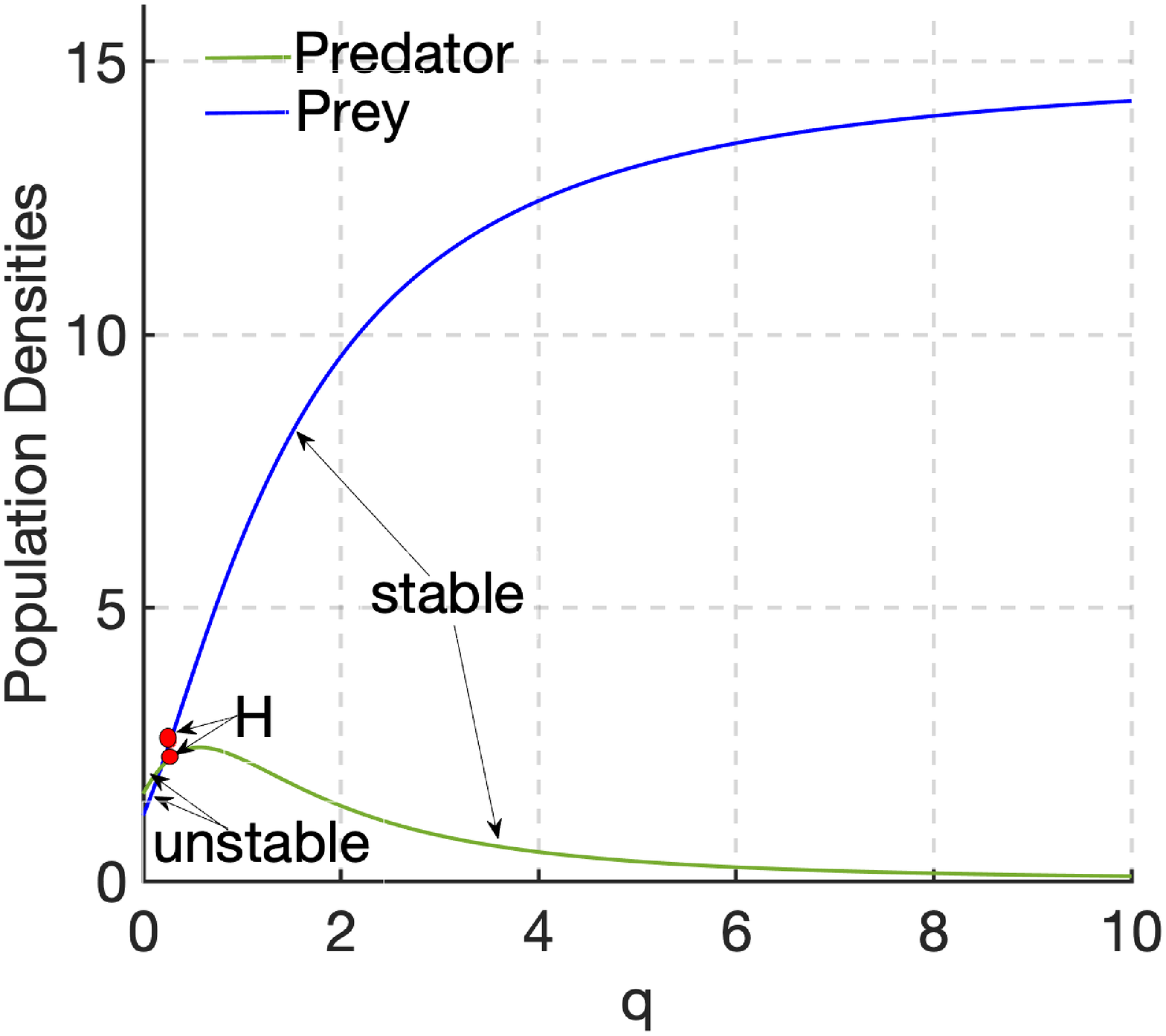}}
\subfigure[]{    
    \includegraphics[scale=.261]{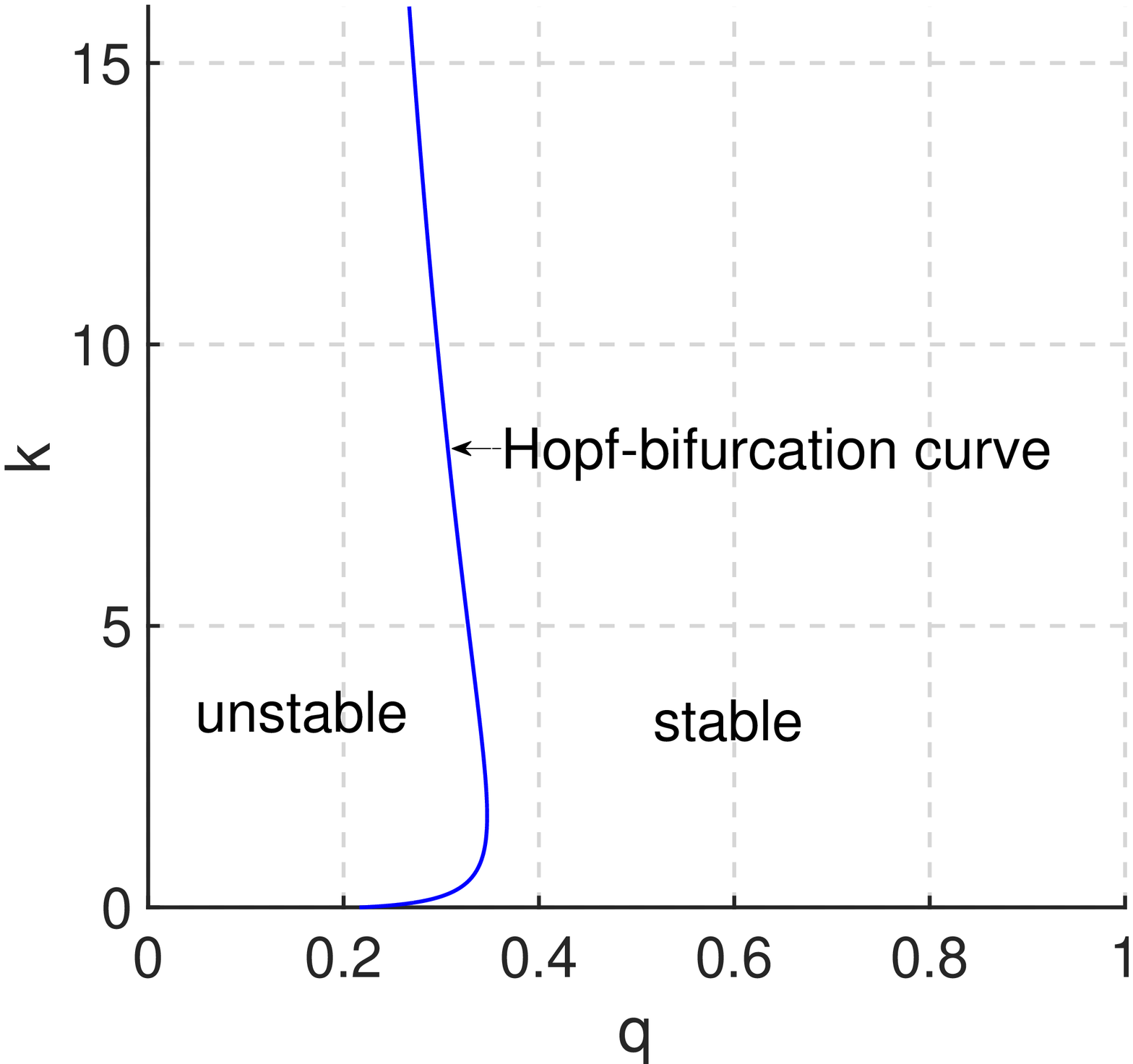}}
\end{center}
\caption{(a) The change in stability of the coexistence equilibrium points as we vary $q$ in the model \eqref{eq:HarvestingEqn}. Hopf-bifurcation occurs at $q=q_h=0.27068$ around $(2.54542,2.24175)$ with first Lyapunov coefficient $L_H=8.38051e^{-03}$. The periodic orbits emitted at the Hopf point is subcritical (b) two-parameter bifurcation diagram in $q-k$ parametric plane. Parameter values used here are from Example \ref{example:Hopf_q}.}
\label{fig:Hopf_HarvestParam_q}
\end{figure}

We now state two conjectures concerning the effect of effort harvesting of predators on the dynamics of model \eqref{eq:HarvestingEqn}.
\begin{conjecture}[Existence of Homoclinic Bifurcation]\label{thm:Harvest-Homoclinic}
Consider the model \eqref{eq:HarvestingEqn} where all  parameters are fixed except $q\geq 0$. There exists $q^*>0$ such that a homoclinic loop occurs when $q=q^*$.
\end{conjecture}
%

\begin{conjecture}[Harvesting-Induced Recovery]\label{thm:harvesting induced recovery}
Consider the predator-prey model given by \eqref{eq:HarvestingEqn}. Then there exists a harvesting effort, $0<c_{1}<q^*<c_{2}$ such that the solution to the prey equation does not go extinct in finite time, and in particular for these levels of predator harvesting the solution can be driven to a coexistence state, if initiated from certain initial conditions.
\end{conjecture}

\begin{remark}
Clearly, increasing the harvesting effort $q$, will cause a lowering of the predator nullcline, bringing down the coexistence equilibrium, see Fig. \ref{fig:Harvest_Recovery_Homoclinic}. For large values of $q$, this could be brought as close to the predator free equilibrium as desired, see Fig. \ref{fig:Hopf_HarvestParam_q}.
\end{remark}

\begin{figure}[!htb]
\begin{center}
\subfigure[]{
    \includegraphics[scale=.46]{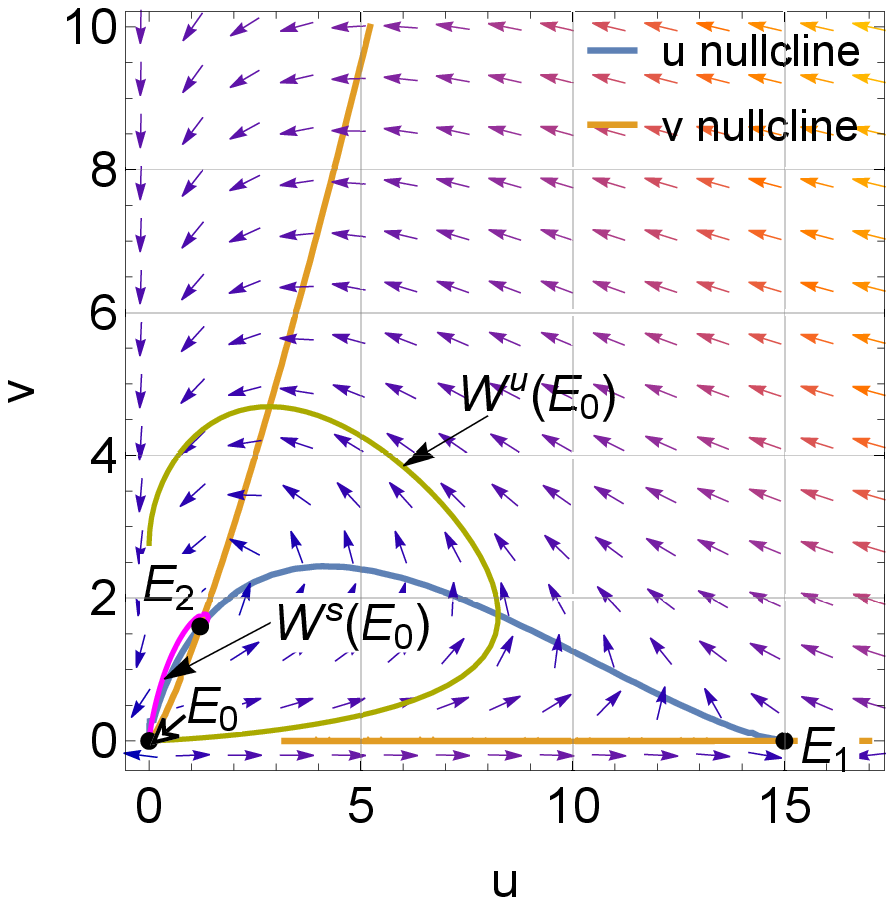}}
\subfigure[]{    
    \includegraphics[scale=.44]{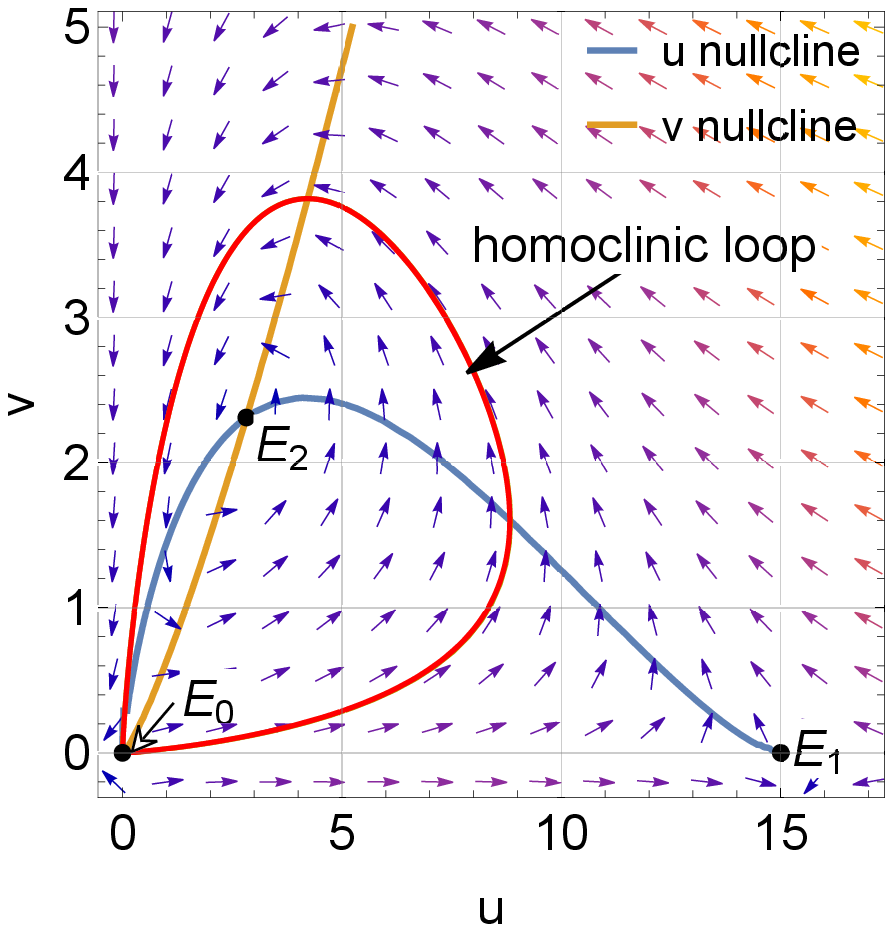}}
\subfigure[]{    
    \includegraphics[scale=.45]{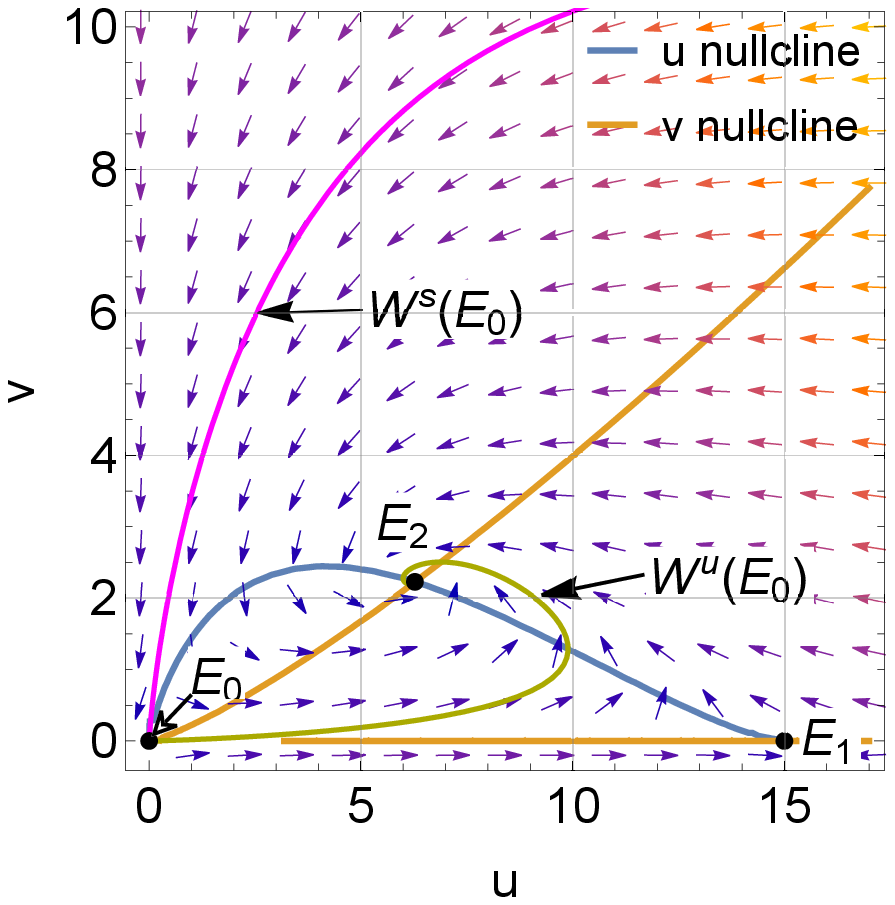}}
\end{center}
\caption{Phase plane portrait of the model \eqref{eq:HarvestingEqn} (a) Stable ($W^s(E_0)$) and unstable ($W^u(E_0)$) manifold with $q=0$, where $W^s(E_0)$ is below $W^u(E_0)$. $E_2$ is unstable and all points are attracted to the predator axis   (b) homoclinic loop when $W^s(E_0)=W^u(E_0)$ for $0.321<q<0.323$. $E_2$ is stable (c) $W^s(E_0)$ is above $W^u(E_0)$ when $q=1$. $E_2$ is stable and all initial data below $W^s(E_0)$ converge to $E_2$ whilst those above are attracted to the predator axis. Parameter values used here are from Example \ref{example:Hopf_q}.}
\label{fig:Harvest_Recovery_Homoclinic}
\end{figure}

%


\section{Conclusions}\label{Sec:Conclusion}
\noindent
In this paper, we have proposed and investigated the rich dynamical behavior of a predator-prey model \eqref{EquationMain}, incorporating the effect of fear, prey herd behavior and mutual interference of predators. The prey herd behavior is governed by a modified Holling type-I functional response that allows for finite time extinction of the prey species due to the non-smoothness at the trivial equilibrium point \cite{RZ71,B12}. The stable manifold ($W^s(E_0)$) of the trivial equilibrium $E_0$ splits the phase plane into two, where solutions with initial conditions above the stable manifold are attracted towards the predator axis in finite time. This posses a problem of non-uniqueness of solutions in backward time.

The fear of predation risk excited by predators can drive a stable prey population species to extinction in finite time, and consequently, to the extinction of the predator population species. To this end please see Theorem  \ref{thm:FDFTE}. We provide numerical justification in Fig. \ref{fig:FDFTE}.
 
Taking the strength of fear of predator $k$ as a bifurcation parameter, we have shown analytically and numerically various local and global bifurcations. We observed from our investigation that fear of predator has the tendency to  stabilize and destabilize a coexistence equilibrium point by producing limit cycles via subcritical Hopf-bifucation.  Biologically, a strong strength of fear can stabilize an unstable coexistence equilibrium of interacting species, see Fig. \ref{fig:Hopf_Main}(a). Also,  with weak strength of fear of predator and certain parameter sets, the stable coexistence between a predator and prey can be destabilized, see Fig. \ref{fig:Hopf_Main}(b). In Fig. \ref{fig:Hopf_Main}, the effects of fear reduces both predator and prey population densities.  From the two-dimensional projections of Hopf-bifurcation curves, we observed a generalized Hopf-bifurcations or Bautin bifurcation which is local and of co-dimension 2 (see Figs. \ref{fig:2param_bifurcation}$(a)-(d)$). 
We obtained saddle-node bifurcations of the model \eqref{EquationMain} as we increase the strength of fear of predator with some appropriate choice of parameters, see Fig. \ref{fig:SN_Main}. There exist a critical strength of fear where the stable and unstable manifold of the trivial equilibrium meet (i.e. homoclinic bifurcation), see Fig. \ref{fig:HomoclinicMain_Curve}(b). This is conjectured in Conjecture \ref{conjecture:Homoclinic_k}. Here we observe that all solutions with initial conditions inside the loop goes to the stable coexistence equilibrium whilst those outside goes to prey extinction in finite time. Hence, these obtained results are interesting and provide further justification that fear of predation risk plays a crucial role in ecosystem balance \cite{PP18}.

We conjecture via Conjecture \ref{thm:harvesting induced recovery} that harvesting of predators can prevent finite time extinction of the prey species and yield persistence of the predator and prey population species. Proving this conjecture would make interesting future work. See Fig. \ref{fig:Harvest_Recovery_Homoclinic}. Additionally, from our numerical simulations in Fig. \ref{fig:Harvest_Recovery_Homoclinic}, when  the unstable manifold of $E_0$ is above the stable manifold of $E_0$ coupled with an unstable coexistence equilibrium point, we observed a homoclinic loop by introducing an external effort dedicated to predator harvesting. Biologically, we are able to stabilize  the predator and prey population species that initiated inside the loop with low harvesting effort. 
 Also, when the external effort dedicated to predator harvesting rate is very high, the prey population species approach its carrying capacity and the predator population species get close to extinction, see Fig. \ref{fig:Hopf_HarvestParam_q}(a). Thus, the predator population species may not survive at a very high harvesting effort. Note, harvesting finds large scale applications in current bio-control applications \cite{LP20}. A full dynamical analysis of \eqref{eq:HarvestingEqn} would make an interesting future endeavor, as the FTE dynamic in predator could counteract with the FTE dynamic in prey, to generate rich dynamical behavior.

Another bio-control application is in pest management where the fear of natural enemy is introduced to drive an invasive pest into extinction. This study should, therefore, prove to be a useful tool in resource management and control. A further interest, which the authors are currently investigating, is the interplay of fear of predator between  aggregating prey  species and  predator interference models, such as those considered in \cite{S97,KRM20, PWB20}.

\section*{Acknowledgements}

\noindent KAF's research is supported by the  Samford Faculty Development Grant (FDG084).
\section*{Appendix}\label{Appendix}
\noindent
We provide numerical simulations to visualize the saddle-node bifurcation of parameters $b,c$ and $d$.
\begin{figure}[!htb]
\begin{center}
\subfigure[]{
    \includegraphics[scale=.171]{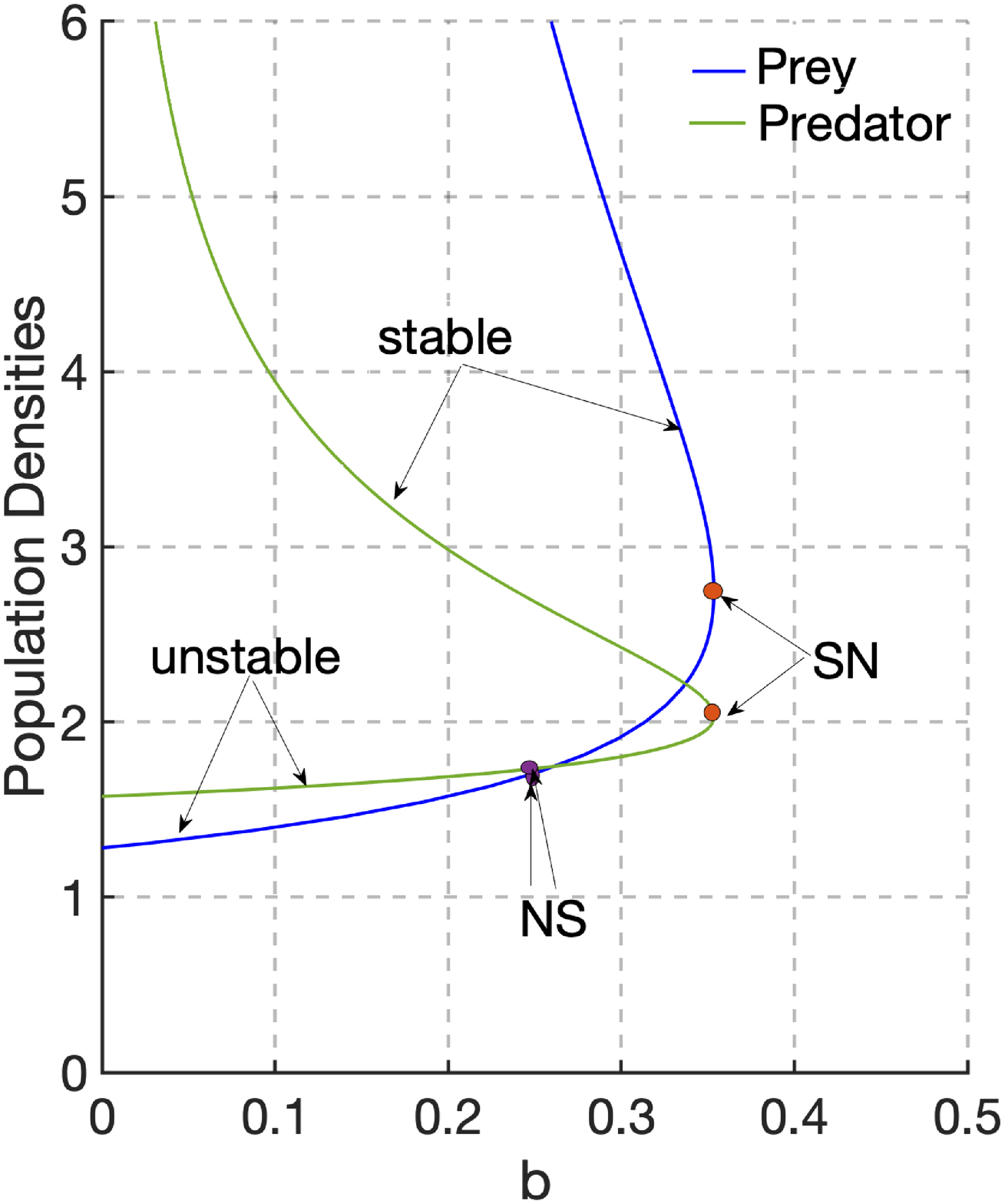}}
\subfigure[]{    
    \includegraphics[scale=.171]{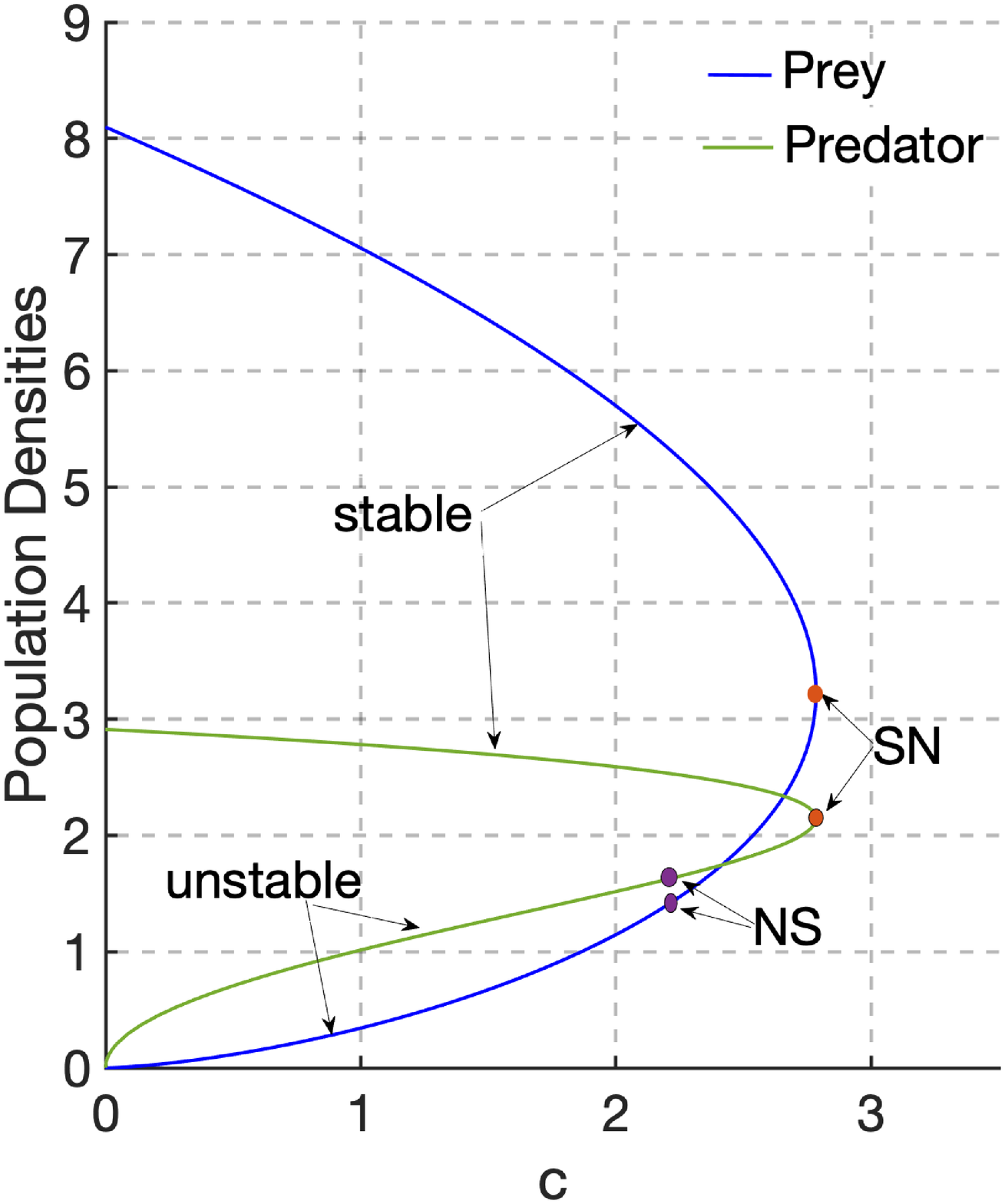}}
\subfigure[]{    
    \includegraphics[scale=.171]{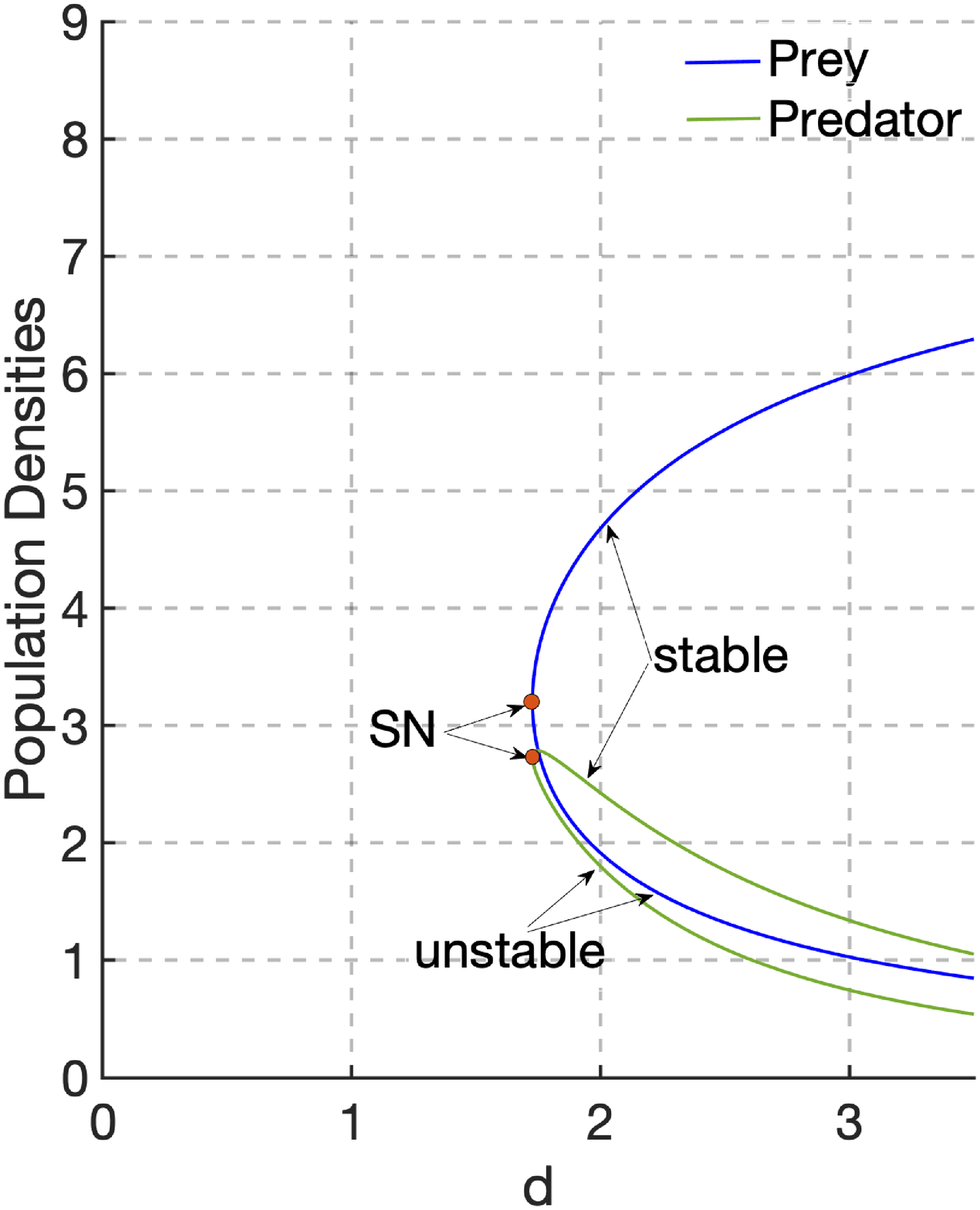}}
\end{center}
\caption{Saddle-node bifurcations of the  model \eqref{EquationMain}  (a)  SN at $b=b_s=0.35355$ and NS at $b=0.24943$  (b) SN at $c=c_s=2.78413$ and NS at $c=2.21579$  (c) SN at $d=d_s=1.72647$. Here, parameters used are given in the caption in Fig. \ref{fig:SN_Main}. (SN: Saddle-node point, NS: Neutral saddle point (not a bifurcation point))}
\label{fig:SN-Appendix}
\end{figure}

\footnotesize


\end{document}